\documentclass{amsart}
\usepackage{amsmath, amsthm, amssymb}
\usepackage{geometry}
\usepackage{mathtools}
\usepackage{graphics}
\usepackage{epsfig}
\usepackage{nicefrac}
\usepackage{mathrsfs}
\usepackage{bm}
\usepackage{stmaryrd}
\usepackage[Symbol]{upgreek}

\numberwithin{equation}{section}

\theoremstyle{plain}
\newtheorem{theorem}{Theorem}[section]
\newtheorem{proposition}{Proposition}[section]

\newtheorem{lemma}{Lemma}[section]

\theoremstyle{remark}
\newtheorem{remark}{Remark}[section]

\theoremstyle{definition}

\newcommand{\Ld}{\mathbf{L}}
\newcommand{\Hd}{\mathbf{H}}

\newcommand{\hone}{H^1(\Omega)}
\newcommand{\hones}{H^1}
\newcommand{\honed}{\Hd^1(\Omega)}
\newcommand{\honeds}{\Hd_0^1}

\newcommand{\hzerod}{\Hd_0^1(\Omega)}
\newcommand{\hzerods}{\Hd_0^1}

\newcommand{\htwo}{\mathbf{H}^2(\Omega)}
\newcommand{\htwos}{\mathbf{H}^2}

\newcommand{\ltwo}{L^{2}(\Omega)}
\newcommand{\ltwos}{L^{2}}
\newcommand{\ltwod}{\Ld^2(\Omega)}
\newcommand{\ltwods}{\Ld^2}

\newcommand{\linfds}{\Ld^{^\infty}}

\newcommand{\lzerotwo}{L_0^2(\Omega)}

\newcommand{\bv}[1]{\mathbf{#1}}

\newcommand{\gradv}[1]{\nabla\bv{#1}}
\newcommand{\diver}[1]{\textsl{div} \, \bv{#1}}

\newcommand{\curl}[1]{\textsl{curl} \, \bv{#1}}

\newcommand{\lp}{\bigl(}
\newcommand{\rp}{\bigr)}
\newcommand{\alp}{\left(}
\newcommand{\arp}{\right)}

\newcommand{\lb}{\left\lbrace}
\newcommand{\rb}{\right\rbrace}
\newcommand{\LN}{\left\|}
\newcommand{\RN}{\right\|}
\newcommand{\dt}{\tau}
\newcommand{\ew}[1]{\mathcal{#1}}
\newcommand{\inc}{\updelta}
\newcommand{\tril}{\,b}
\newcommand{\lvt}{\bm{\upsilon}_h}
\newcommand{\avt}{\bm{\omega}_h}
\newcommand{\ptf}{q_h}
\newcommand{\nunot}{\nu_{_{0}}}

\newcommand{\po}{{\overline{l}}}
\newcommand{\spatialconstantone}{\eta}
\newcommand{\spatialconstanttwo}{\xi}
\newcommand{\peu}{\bv{S}}
\newcommand{\pew}{\mathbf{R}}
\newcommand{\stresstensor}{\sigma}
\newcommand{\momentstresstensor}{\varSigma}
\newcommand{\inertiatensor}{\mathbb{I}}
\newcommand{\inertiamom}{\jmath}
\newcommand{\chartime}{\mathscr{T}}

\newcommand{\polV}{{\mathbb{V} }}
\newcommand{\polQ}{{\mathbb{Q} }}


\begin{document}
\title[Micropolar Navier-Stokes]{The micropolar Navier-Stokes equations: A priori error analysis}

\author[R.H.~Nochetto]{Ricardo H.~Nochetto}
\address[R.H.~Nochetto]{Department of Mathematics and Institute for Physical Science and Technology,
University of Maryland, College Park, MD 20742, USA.}
\email{rhn@math.umd.edu}

\author[A.J.~Salgado]{Abner J.~Salgado}
\address[A.J.~Salgado]{Department of Mathematics, University of Maryland, College Park, MD 20742,
USA.}
\email{abnersg@math.umd.edu}

\author[I.~Tomas]{Ignacio Tomas}
\address[I.~Tomas]{Department of Mathematics, University of Maryland, College Park, MD 20742,
USA.}
\email{ignaciotomas@math.umd.edu}

\thanks{
This work is supported by NSF grants DMS-0807811 and DMS-1109325.
AJS is also partially supported by NSF grant DMS-1008058 and an AMS-Simons Grant.
}



\date{Submitted to M$^3$AS \today.}

\begin{abstract}
The unsteady Micropolar Navier-Stokes Equations (MNSE) are a system of parabolic partial differential equations coupling linear velocity and pressure with angular velocity: material particles have both translational and rotational degrees of freedom. We propose and analyze a first order semi-implicit fully-discrete scheme for the MNSE, which decouples the computation of the linear and angular velocities, is unconditionally stable and delivers optimal convergence rates under assumptions analogous to those used for the  Navier-Stokes equations. With the help of our scheme we explore some qualitative properties of the MNSE related to ferrofluid manipulation and pumping. Finally, we propose a second order scheme and show that it is almost unconditionally stable.
\end{abstract}

\maketitle
\section{The Micropolar Navier-Stokes Equations: Background and Motivations}
\label{sec:Intro}
The Micropolar Navier-Stokes Equations (MNSE) are a system of time-dependent 
partial differential equations that constitutes a framework to describe the dynamics of 
continuum media where the material particles have both translational and rotational degrees of freedom. Consequently, these equations are very attractive for the dynamic description of media subject to distributed couples and polar media in general.
\subsection{The Basic Model}

Let us briefly describe the derivation of the MNSE.
The mathematical modeling of the laws governing the motion of a fluid begins with 
a description of the conservation of mass, linear and angular momentum, which (see \cite{Eringen99} 
or \cite{Luka}) can be written as:
\begin{alignat}{10}
\nonumber
\frac{D\rho}{Dt} &= 0, \\
\label{conserlinearmom}
\rho \frac{D\bv{u}}{Dt} &= \diver{\stresstensor} + \rho \bv{f}, \\
\label{conserangmom}  
\rho \frac{D}{Dt}\alp \bv{\ell} + \bv{x} \times \bv{u} \arp 
&= \rho \bv{g} + \rho \bv{x} \times \bv{f} + \diver{\momentstresstensor} 
+ \bv{x} \times \diver{\stresstensor} + \stresstensor_{\times},   
\end{alignat}
where
$\rho$ is the density;
$\bv{u}$ is the linear velocity;
$\stresstensor \in \mathbb{R}^{3 \times 3}$ is the Cauchy stress tensor;
$\bv{f}$ is the density of external body forces per unit mass;
$\bv{\ell}$ is the angular momentum per unit mass;
$\momentstresstensor \in \mathbb{R}^{3 \times 3}$ is the moment stress tensor;
$\bv{g}$ represents a body source of  moments; and
$(\stresstensor_{\times})_i = \epsilon_{ijk} \stresstensor_{jk}$, where $\epsilon_{ijk}$ is the
Levi-Civita
symbol, i.e., $\epsilon_{ijk} = \frac12(i-j)(j-k)(k-i)$.
As usual, we denote by $D/Dt$ the material derivative.
The physical meaning of the moment stress tensor $\momentstresstensor$ is analogous to 
the stress tensor $\stresstensor$. In other words, given a plane with normal $\nu$, the vector 
$\bv{m} = \momentstresstensor \cdot \nu$
is the moment vector per unit area acting on that plane.

Take the cross product of $\bv{x}$ and \eqref{conserlinearmom} and subtract the result from
\eqref{conserangmom} to obtain a simplified version of the conservation of angular momentum, namely
\begin{equation}
\label{conserpointmom}
\rho \frac{D\bv{\ell}}{Dt} = \rho \bv{g} + \diver{\momentstresstensor} 
+ \stresstensor_{\times}.
\end{equation}
Expressions \eqref{conserangmom} and \eqref{conserpointmom} are usually attributed to
Dahler and Scriven (see \cite{Dahler1} and \cite{Dahler2}) and
have been extensively used by Eringen (see \cite{Eringen99} and \cite{Eringen01})
to develop a general theory of continuum media with
director fields or, more generally, continuum media with microstructure.

In classical continuum mechanics it is usually assumed that the microconstituents do not possess angular momentum and there are no distributed couples. In other words,
$\ell = 0$, $\momentstresstensor = 0$  and $\bv{g} = 0$. Under these assumptions,
\eqref{conserpointmom} implies that the stress tensor $\stresstensor$ is symmetric,
which is the situation generally considered in the literature.
These assumptions are appropriate for most practical applications.
However, this approach is not satisfactory (nor even physical) when, for instance,
the orientability of the microconstituents plays a major role in
the physical process of interest. Classical examples are anisotropic fluids,  liquid polymers,
fluids with rod-like particles, ferrofluids, liquid crystals and polarizable media in general. In these cases a precise description of the moments and rotations associated to the microconstituents of the material is necessary.

In the situation described above, the conservation of angular momentum
\eqref{conserpointmom} needs to be taken explicitly into account which, among other things,
means that it is necessary to propose constitutive relations for $\stresstensor$, $\ell$ 
and $\momentstresstensor$.
Eringen proposed the following (cf.~\cite{EringMicro,Eringen01,Luka}):
\[
  \ell = \inertiatensor \bv{w},
\]
where $\inertiatensor \in \mathbb{R}^{3 \times 3}$ is the so-called microinertia density tensor;
\[
  \stresstensor = (-p + \lambda \diver{u}) \bv{I}
  + \mu (\gradv{}\bv{u}+\gradv{}\bv{u}^T) + \mu_r (\gradv{}\bv{u}-\gradv{}\bv{u}^T) +
  \bv{w}_{\times},
\]
where $p$ is the pressure, $\bv{I} \in \mathbb{R}^{3 \times 3}$ is the identity
tensor, and $(\bv{w}_{\times})_{ij} = \varepsilon_{kij} \bv{w}_k$; and
\[
  \momentstresstensor = \gamma_0 \, \diver{w} \, \bv{I} + \gamma_d (\gradv{}\bv{w}+\gradv{}\bv{w}^T)
  + \gamma_a (\gradv{}\bv{w}-\gradv{}\bv{w}^T).
\]
To further simplify the model we will assume that $\inertiatensor$ is isotropic, so that it can be replaced by a scalar $\inertiamom$, the so-called inertia density. 
To guarantee that the constitutive relationships do not violate the Clausius-Duhem inequality
(see \cite{Luka}), the material constants $\mu$, $\mu_r$, $\gamma_0$, $\gamma_a$ and $\gamma_d$ 
are required to satisfy the following relations:
\begin{equation}
\label{eq:thermconst}
\begin{gathered}
  3 \lambda + 2 \mu \geq 0, \ \ 
  \mu \geq 0, \ \ 
  \mu_r \geq 0, \ \ 
  \gamma_d \geq 0, \ \ 
  \gamma_a + \gamma_d \geq 0, \\
  3\gamma_0 + 2 \gamma_d \geq 0 , \ \ 
  - (\gamma_a + \gamma_d) \leq \gamma_d - \gamma_a \leq (\gamma_a + \gamma_d). 
\end{gathered}
\end{equation}
As a final simplification, we will assume that the fluid is incompressible and has constant density.

Let $\Omega \subset \mathbb{R}^d$  with $d=2$ or $3$ be the domain occupied by the fluid. Replacing these constitutive relationships into \eqref{conserlinearmom} and \eqref{conserpointmom},
we arrive at the MNSE,
\begin{align}\label{eq:microNS}
\left\{
\begin{aligned}
\bv{u}_{t} - (\nu + \nu_r)\Delta \bv{u} + (\bv{u}\cdot\nabla)\bv{u}+\nabla p &= 2 \nu_r
\curl{w} + \bv{f}  \text{,}  \\
\diver{u} &= 0 \text{,} \\
\inertiamom \bv{w}_{t} - (c_a + c_d)\Delta \bv{w} + \inertiamom (\bv{u}\cdot\nabla)\bv{w}- (c_0 +
c_d - c_a) \nabla \diver{w} + 4 {\nu}_r \bv{w} &= 2 \nu_r \curl{u} + \bv{g} \text{,} 
\end{aligned}
\right.
\end{align}
where we implicitly redefined the pressure as $\rho^{-1} p$, and the constants $\nu$, $\nu_r$, $c_a$, $c_d$ and $c_0$ are the kinematic viscosities (i.e. $\mu$, $\mu_r$, $\gamma_a$, $\gamma_d$ and $\gamma_0$ divided by $\rho$, respectively). This system is supplemented with the following initial and boundary conditions
\begin{equation}
\label{IBdata}
\begin{aligned}
  &\bv{u}\rvert_{t=0} = \bv{u}_0, &&\bv{w}\rvert_{t=0} = \bv{w}_0, \\
  &\bv{u}\rvert_{\partial\Omega \times (0,T)} = 0, 
      &&\bv{w}\rvert_{\partial\Omega \times (0,T)} = 0.
\end{aligned}
\end{equation}
The reader is referred to \cite{Luka} for questions regarding existence, uniqueness and
regularity of solutions to \eqref{eq:microNS}-\eqref{IBdata} and related models. The purpose of our work is to propose and analyze numerical techniques for this problem. To simplify notation, in what follows we will set 
\begin{align}
\label{newconstants}
\nunot = \nu + \nu_r , \ \ c_1 = c_a + c_d, \  \  c_2 = c_0 + c_d - c_a , 
\end{align}
and we will assume that $c_1, c_2 > 0$ (see for instance \cite{Luka}) which is consistent with the thermodynamical constraints \eqref{eq:thermconst}.

The MNSE can be regarded as a building block of models that describe the physics of
polarizable media. For instance, Rosensweig (see \cite{Ros97}) described the behavior of ferrofluids subject to a magnetizing field $\bv{h}$
with the MNSE and
\begin{equation}
\label{MagmicroNS}
  \begin{dcases}
    \bv{f} = \mu_0 (\bv{m}\cdot\nabla)\bv{h}, \ \ \ \bv{g} = \mu_0 \bv{m} \times \bv{h}, \\
    \bv{m}_t - \alpha \Delta \bv{m} + (\bv{u}\cdot\nabla)\bv{m} = \bv{w} \times \bv{m} -
    \frac{1}{\chartime} (\bv{m} - \varkappa_0 \bv{h}) & \text{in } \Omega,
  \end{dcases}
\end{equation}
where $\bv{m}$ denotes the magnetization field and $\chartime > 0 $, $\alpha \geq 0$, $\varkappa_0 >
0 $ are material constants. The magnetizing field is assumed to obey the Maxwell equations. The reader is referred to \cite{Ami10} for an analysis of this model.

In addition to applications in smart fluids and polarizable media, there has been
a growing interest on the MNSE in other areas. For instance, they have been used to describe collisional granular flows, where the size of the microconstituents is comparable to the macroscopic scale (\cite{Mit02}) and the frictional interaction between particles is not properly modeled by the classical equations of hydrodynamics. Another application is the modeling of micro and nano flows (\cite{Pap99}), where again the size of the microconstituents is comparable to the ``macroscopic'' scale and the rotational effects cannot be neglected.

The key points of this paper are organized as follows. Section \ref{numericalexperiment} introduces a very simple experiment (ferrofluid pumping) as a motivation for the analysis and numerical implementation of the MNSE. In Section \ref{subsec:energyestiamtes} we recall the basic energy estimates and existence theory for the MNSE. Paragraphs \ref{sub:timedisc} and \ref{finitelementspace} introduce the notation and the basic tools required for the  analysis of the numerical scheme proposed later in Section \ref{sec:Scheme}. Error estimates for the linear and angular velocities are derived in Section \ref{sub:errorestvelocities}, and error estimates for the pressure are derived in Section \ref{errorestimatespressure}. 
We present a formally second order scheme in Section~\ref{sec:2ndorder}, and show that it is almost unconditionally stable, i.e.~it is stable provided the time step is smaller than a constant dependent on the material parameters, but not on the space discretization; see \eqref{eq:conddt2ndorder} for details. Finally, in Section \ref{sec:numericalvalidation}, we provide numerical validation of the error estimates derived earlier.

\subsection{Potential Application: Ferrofluid Pumping by Magnetic Induction}
\label{numericalexperiment}

\begin{figure}
\begin{center}
\includegraphics[scale=0.4]{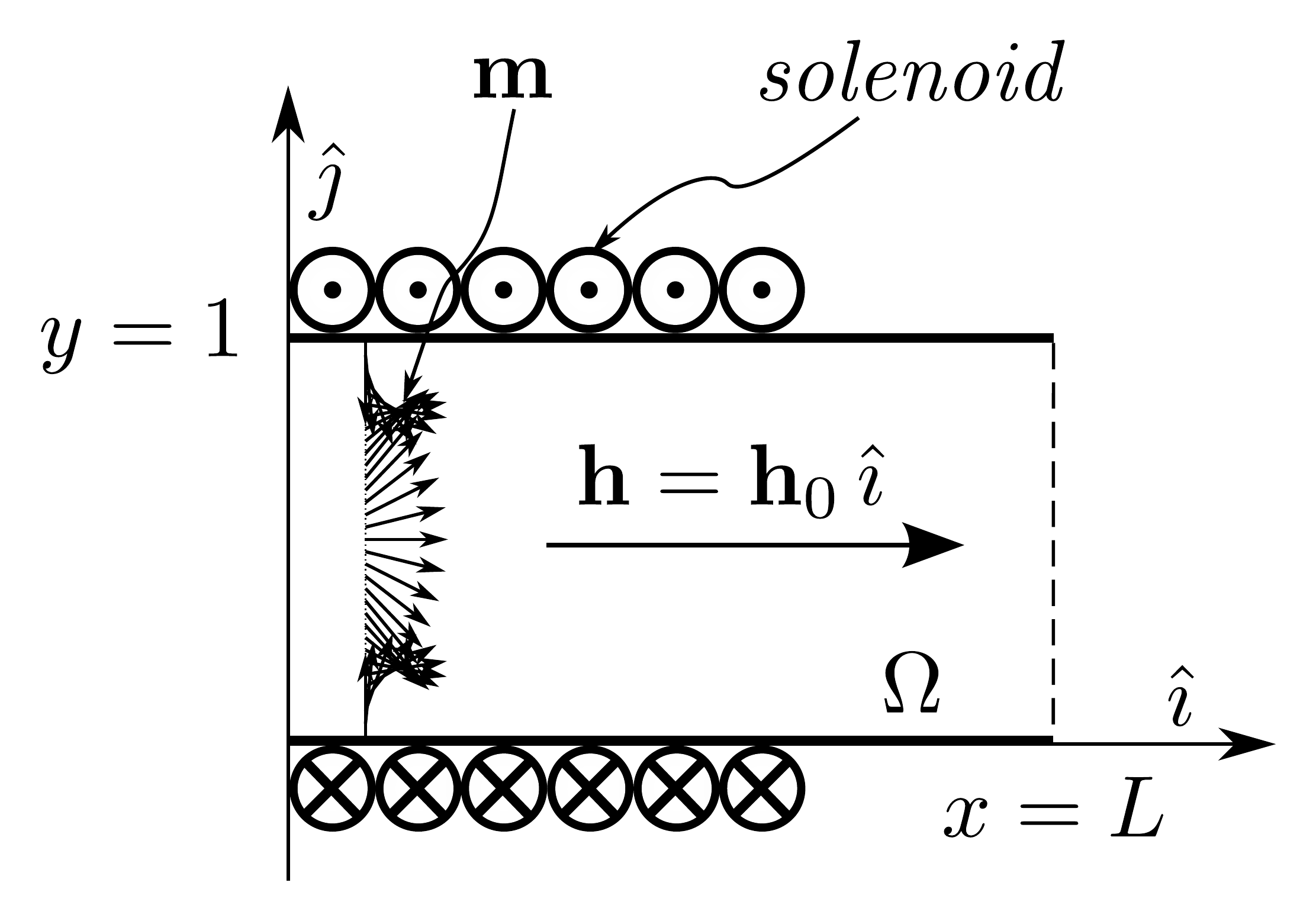}
\end{center}
\caption{Idealized configuration of a ferrofluid pumping experiment. A planar duct with a solenoid that generates a uniform magnetizing field $\bv{h} = h_0 \hat\imath$. Since, $\bv{g} = \mu_0 \bv{m} \times \bv{h}$ (see \eqref{MagmicroNS}), it will produce torque in the regions where $\bv{h}$ and $\bv{m}$ are not collinear. In a real ferrofluid the magnetization vector field $\bv{m}$ would evolve through the channel satisfying the evolution equation \eqref{MagmicroNS} and will try to align with the magnetizing field. However, and as part of an idealized setting, we will assume that the magnetization profile $\bv{m}$ depends only on the $y$-direction.}
\label{figura5}
\end{figure}

To illustrate the differences between the MNSE and the classical Navier-Stokes equations here we propose a setting by means of which it is possible, at least theoretically, to generate fluid motion via a well designed forcing term in the equation of angular momentum. This example is inspired by \cite{Zahn95}, where a ferrofluid is pumped by the actuation of a spatially-uniform sinusoidally time-varying magnetizing field. Another pumping strategy, this time based on a magnetizing field that is varying in space and time, is proposed in \cite{Mao05}.

The idealized setting that we shall consider is depicted in Figure~\ref{figura5}. We assume that our domain is a planar duct of unit height and length $L\geq1$, which is wrapped by a solenoid that generates a uniform magnetizing field $\bv{h} = h_0  \, \hat\imath$, where $\bv{h}_0$ is just a positive constant. From \eqref{MagmicroNS} we infer that $\bv{f}=0$, since the magnetizing field is constant in space. As part of our idealized setting, we disregard the evolution equation in \eqref{MagmicroNS} for the magnetization field, and set $\bv{m}$ to be constant in time and depend only on the vertical variable $y$, i.e.,
\[
  \bv{m} = m_0 (\cos \theta \hat\imath + \sin\theta \hat\jmath) , 
\]
where $m_0$ is just a positive constant, and $\theta = \theta(y)$. Using \eqref{MagmicroNS} we get:
\begin{equation}
  \bv{g} = - \mu_0 m_0 h_0 \sin \theta(y) \, \hat\kappa .
\label{eq:formofg}
\end{equation}

\begin{figure}
\begin{center}
\includegraphics[scale=0.4]{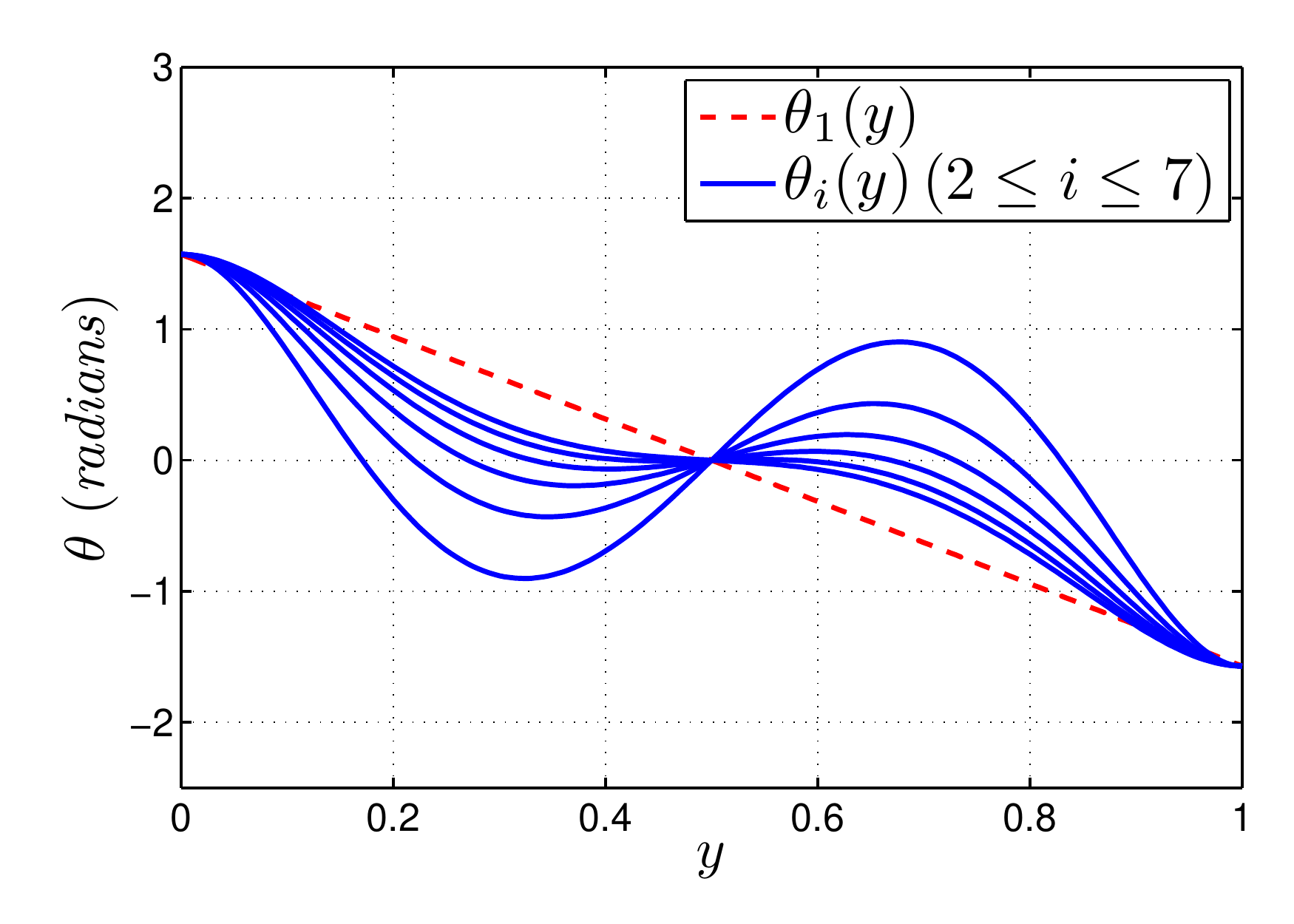}
\caption{Plot of the function $\theta_1(y)$ (dotted line), and the family of functions
$\lb\theta_i(y)\rb_{i=2}^7$ (solid lines). These are used to induce a force in the angular momentum equation. The function $\theta_1$ is a linear interpolation between $\pm\pi/2$ and $\theta_i$, for $i=2,\ldots,7$ are small perturbations of it.}
\label{figura6}
\end{center}
\end{figure}

As reference configuration we will consider a linear interpolation between the points $(0,\pi/2)$ and $(1,-\pi/2)$, that is
\[
  \theta_1(y) = - \pi \alp y - \frac12 \arp.
\]
As perturbations from this reference case we consider, for $i=2,\ldots,7$,
\[
  \theta_i(y) = - \frac{\pi(480 x^5 - 1200 x^4 - 4 x^3 (i^2 + 10 i - 275) + 6 x^2 (i^2 + 10 i - 75)
-
  i^2 - 5 (2 i - 7))}{2 (i^2 + 10i - 35)}.
\]
A plot of these functions is provided in Figure~\ref{figura6}. Notice that they all satisfy
$\theta_i(0) = \pi/2$, $\theta_i(1/2) = 0$ and $\theta_i(1) = -\pi/2$ which we require to model a magnetization field that is perfectly aligned with the magnetizing field at the center of the channel, but is perpendicular to it at the top and bottom walls.

\begin{figure}
\begin{center}
\includegraphics[scale=0.55]{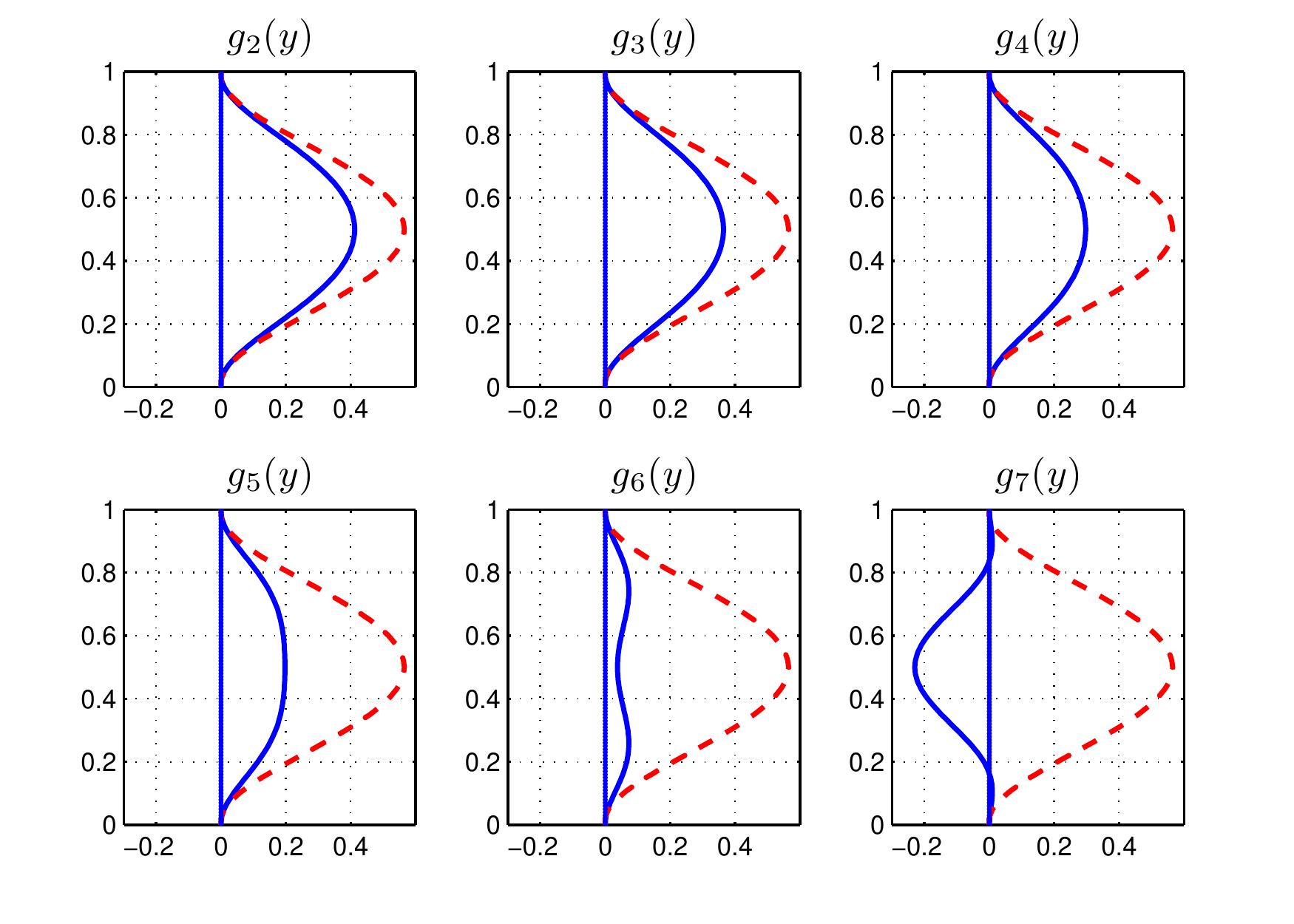}
\caption{Velocity profiles obtained with the forcing terms $\lb\bv{g}_i\rb_{i=2}^7$ (solid lines). For comparison the velocity profile obtained by using $\bv{g}_1$ is also shown (dotted line). The figures show that it is possible to generate linear velocity via appropriate actuation in the angular momentum equation. Notice that, although it is not dramatically different from the others, the forcing term $\bv{g}_7$ induces motion in the opposite direction.}
\label{figura7}
\end{center}
\end{figure}
We assume the fluid is initially at rest, the boundary conditions for the upper and lower part of the duct are no slip, and for the left and right sides of the duct we consider open boundary conditions. We apply the magnetizing field linearly in time, that is we set $\bv{h} = h_0 (t/T) \hat\imath$. We let $L=1$, and the material constants be $\nu = \nu_r = 1$, $c_a = c_d = c_0 = 1$, and $\jmath = 1$. We use a Taylor-Hood finite element discretization of 40 elements in the horizontal and vertical  directions, and a time-step $\dt = 1/50$. The numerical scheme used for this example is the first order method discussed and analyzed in this work. Figure~ \ref{figura7} shows the velocity profiles at time $t=T$
and $x=1$ obtained by setting $\bv{g}$ as in \eqref{eq:formofg}. These results are stable (in the sense that they do not change) with respect to the spatial and temporal discretizations, and length of the channel. However, as it would happen with any physical model, these results can be sensitive to changes in the constitutive parameters. A discussion about the possible influence of the constitutive parameters on the pumping phenomena goes beyond the scope of this paper (see for instance \cite{Rinal02}). 

The results in Figure~\ref{figura7} give an idea about the kind of forces that are necessary in a real ferrohydrodynamic setting, in particular in the case of a spatially uniform and sinusoidal in time magnetizing field as in \cite{Zahn95}. The main observation here is that small variations of the forcing term can yield quite different flow regimes, including flow in the opposite direction, this feature is observed in experiments (cf. \cite{Zahn1995} ). Finally, the reader should be reminded that this is just an idealized setting which illustrates the main pumping mechanism. In real ferrohydrodynamics we cannot set the value of magnetization $\bv{m}$ as we please because $\bv{m}$ is actually determined by the evolution law in \eqref{MagmicroNS}.
\section{Notation and Preliminaries}
\label{preliminariessection}
We shall consider system \eqref{eq:microNS} in an open, bounded, simply connected domain 
$\Omega \subset \mathbb{R}^d$ with $d=2,3$, with a smooth boundary $\partial\Omega$, for a finite interval of time $(0,T)$, and we will denote $\Omega_T = \Omega \times (0,T)$. We use the standard Sobolev spaces $W^k_q(\Omega)$ for $0\leq k \leq \infty$ and $1\leq q \leq \infty$ that consist of functions $f \in L^q(\Omega)$ whose distributional derivatives of order up to $k$ are also in $L^q(\Omega)$. To simplify notation, we set $H^k(\Omega) = W^k_2(\Omega)$, and denote the closure of $\mathcal{C}^\infty_0(\Omega)$ in $H^1(\Omega)$ by $H^1_0(\Omega)$. We denote with bold characters vector valued functions and  their spaces. The scalar product in $L^2(\Omega)$ and $\bv{L}^2(\Omega)$ are indistinctly denoted by $(\cdot,\cdot)$. The subspace of functions in $\ltwo$ with zero mean is denoted by $\lzerotwo$. Whenever $E$ is a normed space, we denote by $\| \cdot \|_E$ its norm.
The space of functions $\phi:[0,T] \rightarrow E$ such that the map 
$(0,T) \ni t \mapsto \|\phi(t)\|_E \in \mathbb{R}$ is $L^q$-integrable is denoted by $L^q(E)$.

We shall make repeated use of the following integration by parts formula for the $\curl{}$
operator:
\begin{equation}\label{intbypartscurl}
(\curl{w},\bv{u}) = (\bv{w},\curl{u}) \quad \forall \bv{u}, \bv{w} \in \hzerod.
\end{equation}
In addition, we recall that the following orthogonal decomposition of $\hzerod$ 
\begin{equation*}
  \|\gradv{u} \|_{\ltwods}^2 = \|\curl{u} \|_{\ltwods}^2 + \|\diver{u} \|_{\ltwods}^2, \quad
  \forall \bv{u} \in \hzerod 
\end{equation*}
holds true (provided $\Omega$ is bounded and simply connected, see for instance \cite{Girault}) which implies
\begin{equation}\label{curlineq}
  \|\curl{u}\|_{\ltwods}^2 \leq \|\gradv{u}\|_{\ltwods}^2 \quad \forall \bv{u} \in \hzerod.
\end{equation} 

We use the following two classical spaces of divergence-free functions (see for instance
\cite{Temam})
\[
  \mathbb{H} = \lb \bv{u} \in \ltwod \ | \ \diver{u} = 0 \ \text{in } \Omega \text{ and }
  \bv{u}\cdot\nu = 0 \text{ on } \partial\Omega \rb,
  \quad
  \mathbb{V} = \hzerod \cap \mathbb{H}.
\]

Henceforth $C$ denotes a generic constant, whose value might change at each
occurrence. This constant might depend on the data of our problem and, when discussing
discretization, its exact solution, but it does not depend on the discretization parameters or the
numerical solution. We denote by $C_p$ the best constant in the Poincar\'e inequality, i.e.,
\[
  \|\bv{u}\|_{\ltwos} \leq C_p\|\nabla \bv{u}\|_{\ltwods} \quad \forall \bv{u} \in \hzerod,
  \quad \ C_p \approx  \textsl{diam}(\Omega).
\]

We will use, as it has become customary, the following trilinear form
\[
  \tril(\bv{u},\bv{v},\bv{w}) = \sum_{i,j} \int_{\Omega}  \bv{u}^{i} \, \bv{v}_{x_i}^j \bv{w}^{j} \,
dx,
  \quad \bv{u},\bv{v},\bv{w}\in\hzerod,
\]
which, as it is well known (cf.~\cite{Temam}), is skew-symmetric whenever the first
argument belongs to $\mathbb{V}$. In addition, we shall use the following, also 
well known, inequalities (see \cite{MarTem}):
\begin{align}
\label{firstineq}
  \tril(\bv{u},\bv{v},\bv{w}) &\leq C
  \|\nabla \bv{u}\|_{\ltwods} \|\nabla \bv{v}\|_{\ltwods} \|\nabla \bv{w}\|_{\ltwods},
  &&\forall \bv{u},\bv{v},\bv{w} \in \hzerod, \\
\label{fourthineq}
  \tril(\bv{u},\bv{v},\bv{w}) &\leq C 
  \|\bv{u}\|_{\linfds} \|\gradv{v}\|_{\ltwods} \|\bv{w}\|_{\ltwods},
  &&\forall \bv{u} \in \htwo, \forall \bv{v}\in \hzerod, \forall \bv{w}\in \ltwod, \\
\label{thirdineq}
  \tril(\bv{u},\bv{v},\bv{w}) &\leq C
  \|\bv{u}\|_{\ltwods} \|\gradv{v}\|_{\ltwods} \|\bv{w}\|_{\linfds},
  &&\forall \bv{u} \in \ltwod, \forall \bv{v}\in \hzerod, \forall \bv{w}\in \htwo, \\
\label{secondineqstar}
  \tril(\bv{u},\bv{v},\bv{w}) &\leq C
  \|\bv{u}\|_{\ltwods} \|\bv{v}\|_{\htwos} \|\nabla \bv{w}\|_{\ltwods},
  &&\forall \bv{u} \in \ltwod, \forall  \bv{v} \in \htwo, \forall \bv{w} \in \hzerod.
\end{align}
\subsection{Energy Estimates and Existence Theorems}
\label{subsec:energyestiamtes}
The stability and error analysis of the scheme that will be proposed in Section \ref{sec:Scheme} is based on energy arguments. Therefore, to gain intuition, let us briefly describe the basic formal energy estimates that can be obtained from \eqref{eq:microNS}. Multiply the linear momentum equation by $\bv{u}$ and the angular momentum equation by $\bv{w}$ and integrate in $\Omega$. Adding both ensuing equations, we obtain
\begin{equation*}
  \frac{1}{2}\frac{d}{dt} \lp \LN\bv{u}\RN_{\ltwods}^2 + \inertiamom  \|\bv{w}\|_{\ltwods}^2 \rp +
  \nunot \LN\gradv{u} \RN_{\ltwods}^2 + c_1 \LN\gradv{w}\RN_{\ltwods}^2
  + c_2 \LN\diver{w}\RN_{\ltwods}^2 + 4 \nu_r \LN \bv{w} \RN_{\ltwods}^2 
  = 4 \nu_r \lp \curl{u},\bv{w} \rp + \lp \bv{f},\bv{u} \rp + \lp  \bv{g},\bv{w} \rp \, ,
\end{equation*}
where the parameters $\nunot$, $c_1$ and $c_2$ were defined in \eqref{newconstants}. Repeated applications of Young's and Poincar\'e's inequalities yield, after integration in time,
\begin{align}
\label{energybound}
\begin{gathered}
\LN \bv{u}(t) \RN_{\ltwods}^2 + \inertiamom  \LN \bv{w}(t) \RN_{\ltwods}^2 + \nu \int_{0}^{t}
\LN\gradv{u}(s)\RN_{\ltwods}^2 \ ds
+ c_1 \int_{0}^{t} \LN\gradv{w}(s)\RN_{\ltwods}^2 \ ds + 2c_2 \int_{0}^{t} \LN \diver{w}
\RN_{\ltwods}^2 ds \leq \\
\leq C_p^2 \int_{0}^{T} \alp \frac{1}{\nu}\|\bv{f}(s)\|_{\ltwods}^2 +
\frac{1}{c_1}\|\bv{g}(s)\|_{\ltwods}^2 \arp ds 
+ \LN \bv{u}_0 \RN_{\ltwods}^2 + \inertiamom  \LN \bv{w}_0  \RN_{\ltwods}^2
\ \ \ \ \forall t \leq T \, ,
\end{gathered}
\end{align}

This formal energy estimate suggests that solutions to \eqref{eq:microNS} are such that
\begin{equation}\label{minregsetting}
\bv{u}\in L^{\infty}(\mathbb{H})\cap L^{2}(\mathbb{V}), \qquad
\bv{w}\in L^{\infty}(\ltwod)\cap L^{2}(\hzerod).
\end{equation}

To obtain an estimate on the pressure, we use a well-known estimate on the right inverse
of the divergence operator (cf.~\cite{Girault,DuranMusch}), i.e.,
\begin{equation}\label{continfsup}
  \beta \|q\|_{\ltwos} \leq \sup_{\bv{v}\in \hzerods} 
  \frac{(q,\diver{v})}{\|\bv{v}\|_{\hzerods}}, \quad \forall q \in \lzerotwo.
\end{equation}
From \eqref{continfsup} and the linear momentum equation in \eqref{eq:microNS} we get
\[
\beta^2 \int_0^T \|p(s)\|_{\ltwos}^2 ds \lesssim 
\int_0^T \left( \|\bv{u}_{t}(s)\|_{\ltwods}^2 + \|\gradv{u}(s)\|_{\ltwods}^2
  + \|\gradv{u}(s)\|_{\ltwods}^4 +  \|\gradv{w}(s)\|_{\ltwods}^2  +  \|\bv{f}\|_{\ltwods}^2 \right)ds, 
\]
so that, to obtain an estimate on the pressure, we must assume $\bv{u} \in L^4(\hzerod)$ and, in addition, we need an estimate on the time derivative of the linear velocity at least in $L^2(\ltwod)$. This is standard for the Navier-Stokes equations. To obtain it we differentiate with respect to time the equations of conservation of linear and angular momentum. Repeating the steps used to obtain \eqref{energybound} we arrive at the desired estimate.

The existence of weak solutions can be summarized as follows.

\begin{theorem}[Existence of weak solutions]
Let $\bv{f},\bv{g} \in L^2(\ltwod)$, $\bv{u}_0 \in \mathbb{H}$ and $\bv{w}_0 \in \ltwod$. Then there exist
$(\bv{u},\bv{w},p) \in L^{\infty}(\mathbb{H}) \times
L^{\infty}(\ltwod) \times \mathcal{D}'(\Omega_T)$
satisfying \eqref{eq:microNS} in the sense of distributions.
Moreover, $\bv{u}$ and $\bv{w}$ satisfy the energy estimate \eqref{energybound}.
\end{theorem}
\begin{proof}
see \cite[Theorem 1.6.1]{Luka}.
\end{proof}

Just like for the Navier-Stokes equations, uniqueness of solutions of the MNSE is an open issue.

\subsection{Time Discretization}
\label{sub:timedisc}
We introduce $K > 0$ to denote the number of steps, define the time-step as
$\dt = T/K >0$ and set $t^k = k\dt$ for $0\leq k \leq K$. For
$\phi : [0,T] \rightarrow E$, with $E$ being a Banach space, we set $\phi^k = \phi(t^k)$.
A sequence will be denoted by  $\phi^\dt = \lb \phi^k\rb_{k=0}^{K}$ and we introduce the following norms:
\[
  \|\phi^\dt \|_{\ell^\infty(E)} = \max_{0 \leq k \leq K} \|\phi^k\|_E,
  \qquad
  \|\phi^\dt \|_{\ell^2(E)} = \alp \sum_{k=0}^K \dt \|\phi^k\|_E^2 \arp^{\nicefrac{1}{2}}.
\]
We define the backward difference operator
\begin{align}\label{backwarddifferenceop}
  \inc\phi^k = \phi^k - \phi^{k-1},
\end{align}
and set $\inc^2 \phi^k = \inc(\inc \phi^k) = \phi^k - 2\phi^{k-1} + \phi^{k-2}$.

Finally, recall the following discrete Gr\"onwall inequality.
\begin{lemma}[Discrete Gr\"{o}nwall]\label{simplifiedgronwall}
Let $a^\dt$, $b^\dt$, $c^\dt$ and $\gamma^\dt$
be sequences of nonnegative numbers such that
$\tau\gamma_k < 1$ for all $k$, and let $g_0 \geq 0$ be so that the following inequality
holds:
\[
  a_K + \dt \sum_{k=0}^{K} b_k \leq \dt \sum_{k=0}^{K} \gamma_k a_k + \dt \sum_{k=0}^{K} c_k + g_0.
\]
Then
\[
  a_K + \dt \sum_{k=0}^{K} b_k \leq
  \alp \tau \sum_{k=0}^{K} c_k + g_0 \arp
  \exp\alp\tau \sum_{k=0}^{K} \sigma_k \gamma_k\arp,  
\]
where $\sigma_k = (1- \tau \gamma_k)^{-1}$.
\end{lemma}
\begin{proof}
See \cite{MR2249024,HeyRann}.
\end{proof}
\subsection{Space Discretization}
\label{finitelementspace}

To construct an approximation of the solution to \eqref{eq:microNS} via Galerkin techniques we introduce two families of finite dimensional spaces, $\{\polV_h\}_{h>0}$ and 
$\{ \polQ_h \}_{h>0}$ with $\polV_h \subset \hzerod$ and,
$\polQ_h \subset \hone \cap \lzerotwo$. The space $\polV_h$ will be used to approximate the linear and angular velocities and $\polQ_h$ to approximate the pressure. We require that these spaces are compatible,
in the sense that they satisfy the LBB condition
\begin{equation}\label{discreteinfsup}
 \inf_{0\neq\ptf \in \mathbb{Q}_h} \sup_{0\neq\lvt \in \mathbb{V}_h} 
  \frac{( \ptf,\diver{}\lvt )}{\|\ptf\|_{\ltwods} \|\nabla\lvt\|_{\ltwods} } \geq \beta^*,
\end{equation}
with $\beta^*$ independent of the discretization parameter $h$. In addition, we require that the
spaces have suitable approximation properties, in other words,
there exists a $\po \in \mathbb{N}$ such that for $m \in[0,\po]$,
\[
  \inf_{\ptf \in \polQ_h } \| q - \ptf \|_{\ltwos} \leq C h^m \| q \|_{H^m}
  \quad \forall q \in H^m(\Omega)\cap\lzerotwo,
\]
\[
  \inf_{ \lvt \in \polV_h }
  \alp \| \bv{v} - \lvt \|_{\ltwods} + h \| \bv{v} - \lvt \|_{\hzerods} \arp
  \leq C h^{m+1} \| \bv{v} \|_{\bv{H}^{m+1}},
  \quad \forall \bv{v} \in \bv{H}^{m+1} \cap \hzerod.
\]
Lastly, we assume that the velocity space $\polV_h$ satisfies the following inverse inequality:
\begin{equation}
\label{eq:inverseineq}
  \|\bv{u}_h\|_{\linfds} \leq C \psi(h) \|\bv{u}_h\|_{\hzerods} \quad \forall \bv{u}_h \in \polV_h,
\end{equation}
where $\psi(h) = (1+|\log(h)|)^{\frac{1}{2}}$ if $d=2$ and $\psi(h) = h^{-\frac{1}{2}}$
if $d=3$.
References \cite{ErnGuermond, Girault} provide a comprehensive list of suitable
choices for these spaces.

For a.e.~$t\in [0,T]$ we define the Stokes projection of $(\bv{u}(t),p(t))$ as the 
pair $(\bv{u}_h(t),p_h(t)) \in \polV_h \times \polQ_h$ that solves
\begin{align*}
\left\{
\begin{aligned}
\nunot ( \gradv{u}_h , \nabla\!\lvt ) - ( p_h, \diver{}\lvt ) &= 
\nunot ( \gradv{u} , \nabla\lvt ) - ( p, \diver{}\lvt ) &&\forall \lvt \in \mathbb{V}_h \\
( \ptf, \diver{u}_h ) &= ( \ptf, \diver{u} ) &&\forall \ptf \in \mathbb{Q}_h  \, .
\end{aligned}
\right.
\end{align*}

In addition, we define the elliptic-like projection of $\bv{w}(t)$ as the function
$\bv{w}_h(t) \in \polV_h$ that solves
\begin{align*}
c_1 ( \gradv{w}_h, \gradv{}\avt ) &+ c_2 ( \diver{w}_h,\diver{}\avt)
+ 4 \nu_r ( \bv{w}_h,\avt ) = \\
&=c_1 ( \gradv{w}, \gradv{}\avt ) + c_2 ( \diver{w},\diver{}\avt )
+ 4 \nu_r ( \bv{w},\avt )
\ \ \  \forall \avt \in \mathbb{V}_h \, .
\end{align*}
The properties of the Stokes and elliptic-like projections are summarized in the following;
see for instance \cite{GuerQuart}.

\begin{lemma}[Properties of projectors]
If
$(\bv{u},\bv{w},p) \in [L^\infty(\bv{H}^2(\Omega)\cap\hzerod)]^{2}
\times L^\infty(\hone\cap\lzerotwo)$, then the Stokes and elliptic-like projectors are stable in dimension $d \leq 3$, i.e., 
\begin{equation}
\label{projectionstability}
  \|\bv{u}_h \|_{L^{\infty}(\linfds \cap \mathbf{W}_{3}^{1})}
  + \|\bv{w}_h \|_{L^{\infty}(\linfds\cap\mathbf{W}_{3}^{1}) }  
  + \| p_h \|_{L^\infty(\hones)} \leq C 
  \left(
  \|\bv{u}\|_{L^{\infty}(\htwos)} + \|p\|_{L^{\infty}(\hones)} 
  + \|\bv{w}\|_{L^{\infty}(\htwos)} \right) \,.
\end{equation}
If, in addition,
$(\bv{u},\bv{w},p) \in [L^\infty(\bv{H}^{\po + 1 }(\Omega)\cap\hzerod)]^{2}
\times L^\infty( H^{\po}(\Omega) \cap\lzerotwo)$, then the projections satisfy the following
approximation properties:
\begin{align}\label{approximab}
\begin{aligned}
  \|\bv{u} - \bv{u}_h \|_{L^{\infty}(\ltwods)} 
  + h \|\bv{u} - \bv{u}_h \|_{L^{\infty}(\honeds)}
  + h \|p - p_h\|_{L^{\infty}(\ltwos)} &
  \leq C h^{\po+1} \spatialconstantone(\bv{u},p) \\
  \|\bv{w} - \bv{w}_h \|_{L^{\infty}(\ltwods)} 
  + h \|\bv{w} - \bv{w}_h \|_{L^{\infty}(\honeds)} 
  & \leq C h^{\po+1} \spatialconstanttwo(\bv{w}),
\end{aligned}
\end{align}
where $\spatialconstantone(\bv{u},p) = \|\bv{u}\|_{L^{\infty}(\mathbf{H}^{\po+1})} 
+\|p\|_{L^{\infty}(H^{\po})}$ and $\spatialconstanttwo(\bv{w}) =
\|\bv{w}\|_{L^{\infty}(\mathbf{H}^{\po+1})}.$
\end{lemma}

We introduce the trilinear form $b_h: [\hzerod]^3 \rightarrow \mathbb{R},$
\[
  \tril_h(\bv{u},\bv{v},\bv{w}) = \tril( \bv{u},\bv{v},\bv{w} ) + 
  \frac12 (\diver{u},\bv{v}\cdot\bv{w}), \quad \forall \ \bv{u},\bv{v},\bv{w} \in \hzerod,
\]
and recall that it is consistent, i.e., $\tril_h(\bv{u},\bv{v},\bv{w}) = \tril(\bv{u},\bv{v},\bv{w})$ whenever $\bv{u} \in \polV$, and skew-symmetric
\[
  \tril_h(\bv{u},\bv{v},\bv{v}) = 0,
\]
for all $\bv{u}, \bv{v} \in \hzerod$.
This form satisfies estimates similar to \eqref{firstineq}--\eqref{secondineqstar}, namely
\begin{equation}
\label{firstdiscineq}
\begin{aligned}
  \tril_h(\bv{u}_h,\bv{v}_h,\bv{w}_h) &\leq C
  \|\gradv{u}_h\|_{\ltwods} \|\gradv{v}_h\|_{\ltwods} \|\gradv{w}_h\|_{\ltwods}, 
  &&\forall \bv{u}_h, \bv{v}_h,\bv{w}_h \in \polV_h, \\
  \tril_h(\bv{u}_h,\bv{v},\bv{w}_h) &\leq C
  \|\bv{u}_h\|_{\ltwods} \|\bv{v}\|_{\htwos} \|\gradv{w}_h\|_{\ltwods},
  &&\forall \bv{u}_h,\bv{w}_h \in \mathbb{V}_h, \ \forall \bv{v} \in \htwo,
\end{aligned}
\end{equation}
and
\begin{equation}
\label{fifthdiscineq}
  \tril_h(\bv{u}_h,\bv{v}_h,\bv{w}_h) \leq C
  \|\bv{u}_h\|_{\ltwods} \|\bv{v}_h \|_{(\linfds \cap \mathbf{W}_{3}^{1})}
  \|\gradv{w}_h\|_{\ltwods}, \quad \forall \bv{u}_h, \bv{v}_h, \bv{w}_h \in \mathbb{V}_h.
\end{equation}
Since, by assumption, the space $\polV_h$ satisfies the inverse inequality
\eqref{eq:inverseineq}, then for $d=3$,
\begin{align}
\label{additionalineq}
\begin{aligned}
\tril_h(\bv{u}_h,\bv{v}_h,\bv{w}_h) 
&\leq C h^{-\frac{1}{2}} 
\|\bv{u}_h\|_{\ltwods} \|\gradv{v}_h\|_{\ltwods} \|\gradv{w}_h\|_{\ltwods}
\quad &&\forall \bv{u}_h, \bv{v}_h, \bv{w}_h \in \mathbb{V}_h \\
\tril_h(\bv{u}_h,\bv{v}_h,\bv{w}_h) 
&\leq C h^{- \frac{1}{2}} \, \|\gradv{u}_h\|_{\ltwods} \|\gradv{v}_h\|_{\ltwods} 
\|\bv{w}_h\|_{\ltwods} \quad &&\forall \bv{u}_h, \bv{v}_h, \bv{w}_h \in \mathbb{V}_h 
\end{aligned}
\end{align} 
\section{Description of the First Order Scheme}
\label{sec:Scheme}
To the best of our knowledge, the only work that is concerned with the construction and analysis of a scheme for the MNSE is \cite{ort08}, where a fully discrete penalty projection method for this system is developed and analyzed, and a suboptimal convergence rate is derived.  Our scheme instead possesses optimal approximation properties and requires the solution of a saddle point problem at each time step, which can be done efficiently. However, it can be easily modified to decouple the linear velocity and pressure via an incremental projection method, while maintaining optimal orders of convergence. For brevity this will not be included.

Let us now describe the scheme. The scheme computes
$\{\bv{U}_h^\dt,\bv{W}_h^\dt,P_h^\dt\} \subset \polV_h^2 \times \polQ_h$ meant to approximate, at each time step, the linear and angular velocities and the pressure.
We initialize the scheme by setting
\begin{equation}
\label{eq:Init}
  (\bv{U}_h^0,P_h^0) = (\bv{u}_h^0,p_h^0), \qquad
  \bv{W}_h^0 = \bv{w}_h^0,
\end{equation}
that is, we compute the Stokes and elliptic-like projections of the initial data.

\begin{remark}[Initialization]
The initialization step \eqref{eq:Init} requires that the initial data is regular enough so that the projections are well defined, which from now on we will assume. If this is not the case, \eqref{eq:Init} must be modified and, say, take $L^2$-projections. The analysis below must be accordingly adjusted to take this into account (cf.~\cite{HeyRann}).
\end{remark}

After initialization, for $k=1,\ldots,K$, we march in time in two steps: \\

\noindent \textbf{Linear Momentum:} Compute $(\bv{U}_h^k,P_h^k) \in \polV_h \times \polQ_h$, solution of
\begin{subequations}
\label{diffeqfullydiscv}
  \begin{align}
    \label{subeq:lmom}
    \alp \tfrac{ \inc\bv{U}_h^{k}}\dt, \lvt \arp + \nunot \lp \gradv{U}_h^{k},\gradv{}\lvt \rp
+ \tril_h\lp\bv{U}_h^{k-1},\bv{U}_h^{k},\lvt\rp - \lp P_h^{k}, \diver{}\lvt \rp  
&= 2 \nu_r \lp \curl{W}_h^{k-1},\lvt \rp + \lp \bv{f}^{\,k}, \lvt \rp \, , \\
    \label{subeq:div} \lp \ptf, \diver{U}_h^k\rp &= 0 \, , 
  \end{align}
\end{subequations}
for all $ \lvt \in \mathbb{V}_h$, $\ptf \in \mathbb{Q}_h$.
  
\noindent \textbf{Angular Momentum:} Find $\bv{W}_h^k \in \polV_h$ that solves
\begin{multline} 
\inertiamom  \alp \frac{\inc\bv{W}_h^{k}}\dt,\avt \arp 
+ c_1 \lp \gradv{W}_h^{k}, \nabla \avt \rp 
+ \inertiamom \tril_h\lp\bv{U}_h^{k},\bv{W}_h^{k},\avt\rp + \\
+ c_2 \lp \diver{W}_h^{k},\diver{}\avt \rp + 4 {\nu}_r \lp \bv{W}_h^{k}, \avt \rp  
= 2\nu_r \lp \curl{U}_h^{k}, \avt \rp + \lp \bv{g}^{k},\avt \rp \, , 
\label{diffeqfullydiscw}
\end{multline}
for all $\avt \in \mathbb{V}_h$. \\

Notice that we have decoupled the linear and angular momentum equations by time-lagging of the variables. This scheme is unconditionally stable, as the following result shows.

\begin{proposition}[Unconditional stability of the first order scheme] \label{prop:schstable}
The sequence $\{\bv{U}_h^\dt,\bv{W}_h^\dt,P_h^\dt\} \subset [\polV_h]^2 \times \polQ_h$,
solution of \eqref{diffeqfullydiscv}--\eqref{diffeqfullydiscw}, satisfies
\begin{align}
\label{eq:stabilityestimate}
\begin{split}
\|\bv{U}_h^{K}\|_{\ltwods}^2 &+ (\inertiamom +  4 {\nu}_r \dt ) \|\bv{W}_h^{K}\|_{\ltwods}^2
+ \sum_{k=1}^K \alp \|\inc\bv{U}_h^{k}\|_{\ltwods}^2 + \inertiamom \|\inc\bv{W}_h^{k}\|_{\ltwods}^2
\arp
+ \sum_{k=1}^K \dt \alp \nu \|\gradv{U}_h^{k}\|_{\ltwods}^2 
+ \dt c_1 \|\gradv{W}_h^{k}\|_{\ltwods}^2 \arp \\
&+ 2 \sum_{k=1}^K \dt c_2 \|\diver{W}_h^{k}\|_{\ltwods}^2
\leq \sum_{k=1}^K \dt \alp \frac{C_p^2 \nu_r }{\nu} \|\bv{f}^{\,k}\|_{\ltwods}^2  
+ \frac{C_p^2 \nu_r }{c_1} \|\bv{g}^{k}\|_{\ltwods}^2 \arp 
+ \|\bv{U}_h^{0}\|_{\ltwods}^2 + (\inertiamom + 4 \nu_r \dt) \|\bv{W}_h^{0}\|_{\ltwods}^2.
\end{split}
\end{align}
\end{proposition}
\begin{proof}
Set $\lvt = 2 \dt \bv{U}_h^{k}$ in \eqref{diffeqfullydiscv} and $\avt = 2 \dt \bv{W}_h^{k}$
in \eqref{diffeqfullydiscw}, respectively, and add the results.
Use the identity $2a(a-b) = a^2 - b^2 + (a-b)^2$, the integration by parts formula
\eqref{intbypartscurl}, estimate \eqref{curlineq} and Young's inequality to obtain
\begin{align}
\label{auxexp}
\begin{aligned}
 \|\bv{U}_h^{k}\|_{\ltwods}^2 &+ (\inertiamom +  4 {\nu}_r \dt ) \|\bv{W}_h^{k}\|_{\ltwods}^2
+ \|\inc\bv{U}_h^{k}\|_{\ltwods}^2 + \inertiamom \|\inc\bv{W}_h^{k}\|_{\ltwods}^2
+ \dt \nu \|\gradv{U}_h^{k}\|_{\ltwods}^2 
+ \dt c_1 \|\gradv{W}_h^{k}\|_{\ltwods}^2
+ 2 \dt c_2 \|\diver{W}_h^{k}\|_{\ltwods}^2 \\
&\leq \|\bv{U}_h^{k-1}\|_{\ltwods}^2 + (\inertiamom + 4 \nu_r \dt) \|\bv{W}_h^{k-1}\|_{\ltwods}^2 +
\frac{C_p^2 \nu_r \dt}{\nu} \|\bv{f}^{\,k}\|_{\ltwods}^2  
+ \frac{C_p^2 \nu_r \dt}{c_1} \|\bv{g}^{k}\|_{\ltwods}^2.
\end{aligned}
\end{align}
Adding over $k$ we obtain the desired estimate \eqref{eq:stabilityestimate}.
\end{proof}

\section{A Priori Error Analysis}
\label{sec:erroranalysis}
Here we perform an error analysis of scheme \eqref{diffeqfullydiscv}--\eqref{diffeqfullydiscw}
and show that this method has optimal convergence properties. The analysis is based on 
energy arguments and hinges on the unconditional stability result of Proposition~\ref{prop:schstable}. The arguments used are rather standard for the Navier-Stokes equations, the main novelty and difficulty being the coupling with the angular momentum equation, which requires lengthy and careful computations.

We shall assume, for the sake of simplicity,
that the solution to \eqref{eq:microNS}--\eqref{IBdata} satisfies:
\begin{equation}\label{initialregass}
\bv{u}, \bv{w} \in \mathcal{C}^1([0,T],\Hd^{\po+1}(\Omega)) \ \ \ \text{and} \ \ \ \bv{u}_{tt},
\bv{w}_{tt} \in L^2([0,T],\ltwod) .
\end{equation}
These assumptions will be enough to derive optimal convergence rates for the linear and angular velocities. If we want to do the same with the pressure we will require the additional regularity:
\begin{align}
\label{additionalreg}
\bv{u}_{tt}, \bv{w}_{tt} \in \mathcal{C}([0,T],\Hd^{\po+1}(\Omega)) . 
\end{align} 
These assumptions are standard in the error analysis of incompressible flows (cf. \cite{MarTem}). 

The first step in the error analysis is to analyze the consistency of the method. To do so, we proceed as it is customary in the analysis of evolutionary problems (cf.~\cite{MR2249024}) and split the errors
\[
\bv{E}^k = \bv{u}^k-\bv{U}_h^k, \qquad
\ew{E}^k = \bv{w}^k-\bv{W}_h^k, \qquad
e^k  = p^k-P_h^k,
\]
into the so-called interpolation and approximation errors via the Stokes and elliptic projections
of \S\ref{finitelementspace}, i.e.,
\begin{align}
\label{wheelerdecomp2}
\begin{aligned}
&\peu^k = \bv{u}^k-\bv{u}_h^k, \ \  
&\pew^k = \bv{w}^k-\bv{w}_h^k, \ \ \ \ 
&r^k = p^k-p_h^k \, , \\
&\bv{E}_{h}^k = \bv{u}_h^k-\bv{U}_h^k, \ \ 
&\ew{E}_h^k = \bv{w}_h^k-\bv{W}_h^k, \ \ \ \ 
&e_{h}^k = p_h^k-P_h^k.
\end{aligned}
\end{align}
The interpolation errors $(\bv{S}^\dt,\bv{R}^\dt,r^\dt)$ are controlled by means of \eqref{approximab}, so that the next step is to derive an energy estimate for the approximation errors $(\bv{E}_h^\dt,\mathcal{E}_h^\dt,e_h^\dt)$ which is a slight variation of that one obtained for $(\bv{U}_h^\dt,\bv{W}_h^\dt,P_h^\dt)$ in \eqref{eq:stabilityestimate}. 
\subsection{Error estimates for the Linear and Angular Velocities}
\label{sub:errorestvelocities}
The approximation errors $(\bv{E}_h^\dt,\mathcal{E}_h^\dt,e_h^\dt)$ satisfy the following energy identity:
\begin{align}
\label{globalerrorstructure}
\begin{split}
\|\bv{E}_{h}^{k}\|_{\ltwods}^2 &+ (\inertiamom +8 \dt {\nu}_r) \|\ew{E}_{h}^{k}\|_{\ltwods}^2 
- \|\bv{E}_{h}^{k-1}\|_{\ltwods}^2 - \inertiamom \|\ew{E}_{h}^{k-1}\|_{\ltwods}^2
+ 2 \dt \nunot \|\gradv{E}_{h}^{k}\|_{\ltwods}^2 \\
&+ 2 \dt c_1 \|\nabla\ew{E}_{h}^{k}\|_{\ltwods}^2 
+ \|\inc\bv{E}_{h}^{k}\|_{\ltwods}^2 + \|\inc\ew{E}_{h}^{k}\|_{\ltwods}^2 
+ 2 \dt c_2 \|\diver{}\ew{E}_{h}^{k}\|_{\ltwods}^2  = \sum_{i=1}^6 A_i 
\end{split}
\end{align}
with
\begin{align*}
\begin{split}
&A_1 = 2 \dt \tril_{h}\!\lp \bv{U}_h^{k-1},\bv{U}_h^{k}, \bv{E}_{h}^{k} \rp 
- 2 \dt \tril_{h}\!\lp\bv{u}^k,\bv{u}^k,\bv{E}_{h}^{k}\rp \\
&A_2 = 2\inertiamom\dt\tril_{h}\! \lp \bv{U}_h^{k},\bv{W}_h^{k}, \ew{E}_{h}^{k} \rp 
- 2\inertiamom\dt\tril_{h}\!\lp\bv{u}^k,\bv{w}^k,\ew{E}_{h}^{k}\rp \\
&A_3 = 4\dt \nu_r \lp \curl{w}^{k} - \curl{W}_h^{k-1},\bv{E}_{h}^{k} \rp \\
&A_4 = 4 \dt \nu_r \lp \curl{E}^k,\ew{E}_{h}^{k}\rp \\
&A_5 = - 2 \lp \inc\bv{S}^k , \bv{E}_{h}^{k} \rp 
- 2 \inertiamom \lp\inc\bv{R}^k ,\ew{E}_{h}^{k}\rp \\
&A_6 = 2 \dt\lp\mathcal{R}_\bv{u}^k,\bv{E}_{h}^{k}\rp
+ 2 \inertiamom \dt \lp \mathcal{R}_\bv{w}^k,\ew{E}_{h}^{k}\rp,
\end{split}
\end{align*}
where $\mathcal{R}_\bv{u}^k$ and $\mathcal{R}_\bv{w}^k$ are integral representations 
of Taylor remainders (see for instance \cite{MarTem}), i.e.
\begin{align}
\label{truncationerrors}
\begin{split}
\mathcal{R}_\bv{u}^k = \frac{1}{\dt} \int_{t^{k-1}}^{t^{k}} (t^{k-1}-s) \bv{u}_{tt}(s) \, ds \, \ \
\text{and} \ \ \ \mathcal{R}_\bv{w}^k = \frac{1}{\dt} \int_{t^{k-1}}^{t^{k}} (t^{k-1}-s)
\bv{w}_{tt}(s) \, ds \, .
\end{split}
\end{align}
The main difficulty, and our focus from now on, is to estimate the residual terms $A_i$, $i=1,\ldots,6$. 
\begin{theorem}[Error estimate on velocities] \label{maintheorem}
Assume \eqref{initialregass}, then
\begin{align}
\label{finalLinfL2error}
\|\bv{E}^{\dt}\|_{L^{\infty}(\ltwods)} + 
\|\ew{E}^{\dt}\|_{L^{\infty}(\ltwods)} +
h \left( \|\nabla\bv{E}^{\dt}\|_{L^2(\ltwods)}
    + \|\nabla\ew{E}^\dt\|_{L^2(\ltwods)}
    \right) \leq C( \dt + h^{\po+1} )
\end{align}
whenever
\begin{align}\label{Kvalue}
\dt \leq \frac{1}{K} \ \ \text{with} \ \ 
K \simeq \text{ max} \lb{\frac{M}{\nu_0},\frac{\mathcal{M}\inertiamom^2}{c_1},
\frac{M\inertiamom^2}{c_1},\frac{\nu_r^2}{c_1},\frac{\nu_r^2}{\nu_0}} \rb \, , 
\end{align}
where $M$ and $\mathcal{M}$ satisfy
\begin{align}
\label{eq:Mestiamtes}
\begin{split}
\sup_{\Omega_T} \lp \|\gradv{u}_h\|_{\Ld^3}+ \|\bv{u}_h\|_{\linfds}\rp^2 
+ \sup_{\Omega_T} |\bv{u}|^2 &\leq M < \infty \, , \\
\sup_{\Omega_T} \lp \|\gradv{w}_h\|_{\Ld^3}+ \|\bv{w}_h\|_{\linfds}\rp^2 + \sup_{\Omega_T}
|\bv{w}|^2  &\leq \mathcal{M} < \infty \, .  
\end{split}
\end{align}
\end{theorem}
\begin{proof}
It suffices to provide bounds for the terms $A_i$ above and employ the
discrete Gr\"onwall lemma. To begin with, notice that
\begin{align}
\label{convectivemanipulation}
\begin{split}
&\tril_h(\bv{U}_h^{k-1},\bv{U}_h^{k},\lvt) - \tril_h(\bv{u}^k,\bv{u}^k,\lvt)=
- \tril_h(\inc \bv{u}^{k},\bv{u}^k,\lvt) 
- \tril_h(\bv{u}^{k-1},\bv{S}^k,\lvt)\\  
&\ \ \ \ - \tril_h(\bv{S}^{k-1},\bv{u}_h^k, \lvt) 
- \tril_h(\bv{E}^{k-1},\bv{u}_h^k, \lvt)
+ \tril_h(\bv{E}_h^{k-1}, \bv{E}_h^{k}, \lvt) 
- \tril_h(\bv{u}_h^{k-1}, \bv{E}_h^{k}, \lvt)  \ \ \forall \lvt \in \mathbb{V}_h \, , 
\end{split} 
\end{align}
and
\begin{align}
\label{manipconvectivespin}
\begin{split}
\tril_h\lp\bv{U}_h^{k},\bv{W}_h^{k},\avt \rp
&- \tril_h\lp\bv{u}^k,\bv{w}^k,\avt\rp = 
- \tril_h\lp\bv{u}^{k},\bv{R}^{k}, \avt \rp 
+ \tril_h\lp\bv{E}_h^{k},\ew{E}_h^{k},\avt \rp \\
&- \tril_h\lp\bv{E}_h^{k},\bv{w}_h^{k},\avt \rp
- \tril_h\lp\bv{S}^{k},\bv{w}_h^{k},\avt \rp
- \tril_h\lp\bv{u}_h^{k},\ew{E}_h^{k},\avt \rp \ \ \ \ \  \forall \avt \in \mathbb{V}_h \, . 
\end{split} 
\end{align}

Set $\lvt = 2\dt \bv{E}_{h}^{k}$ in (\ref{convectivemanipulation}). Since $\tril_h$
is skew-symmetric the last two terms vanish, and we can rewrite $A_1$ as:
\begin{align}
\nonumber
\begin{split}
A_1 &= - 2 \dt \tril_{h}\!\lp \inc \bv{u}^k,\bv{u}^k,\bv{E}_{h}^{k}\rp 
- 2 \dt \tril_{h}\!\lp\bv{u}^{k-1},\peu^k,\bv{E}_{h}^{k}\rp
- 2 \dt\tril_h(\bv{S}^{k-1},\bv{u}_h^k, \bv{E}_{h}^{k}) 
- 2 \dt \tril_h(\bv{E}_h^{k-1},\bv{u}_h^k, \bv{E}_{h}^{k}) \\
&= A_{11} + A_{12} + A_{13} + A_{14} \, . 
\end{split}
\end{align}
The functions $\inc \bv{u}^k$ and $\bv{u}^{k-1}$ are solenoidal so that the consistency of $\tril_h$ yields control on $A_{11}$ and $A_{12}$:
\begin{align*}
A_{11} &= 2 \dt \tril_{h}\lp \inc \bv{u}^k,\bv{E}_{h}^{k}, \bv{u}^k\rp
\leq 2 \dt \|\inc \bv{u}^k\|_{\ltwods} \|\gradv{E}_{h}^{k}\|_{\ltwods} 
\|\bv{u}^k\|_{\linfds}
\leq \frac{\nunot \dt}{9}  \|\gradv{E}_{h}^{k}\|_{\ltwods}^2 
+ \frac{9 M \dt}{\nunot}\|\inc\bv{u}^k\|_{\ltwods}^2 \\
A_{12} &= - 2 \dt \tril_{h}\lp\bv{u}^{k-1},\peu^k,\bv{E}_{h}^{k}\rp 
\leq 2 \dt \ \|\bv{u}^{k-1}\|_{\linfds} \|\gradv{E}_{h}^{k}\|_{\ltwods} \|\peu^k\|_{\ltwods} 
\leq \frac{\nunot \dt}{9} \|\gradv{E}_{h}^{k}\|_{\ltwods}^2 + \frac{9 M \dt}{\nunot}
\|\peu^k\|_{\ltwods}^2,
\end{align*}
where we have used \eqref{thirdineq} and \eqref{fourthineq}. By \eqref{initialregass}, we deduce 
\begin{align}\label{incrementestimate}
\|\inc\bv{u}^k\|_{\ltwods}^2 \leq \dt \int_{t_{k-1}}^{t_k} \|\bv{u}_t\|_{\ltwods}^2 \, dt.
\end{align}
The terms $A_{13}$ and $A_{14}$ can be estimated via \eqref{fifthdiscineq} as follows:
\begin{align*}
A_{13} + A_{14} &\leq \frac{2 \nunot \dt}{9} \|\gradv{E}_{h}^{k}\|_{\ltwods}^2 
+ \frac{9 M \dt^2}{\nunot} \|\peu^{k-1}\|_{\ltwods}^2 
+ \frac{9 M \dt^2}{\nunot} 
\|\bv{E}_{h}^{k-1}\|_{\ltwods}^2.
\end{align*}

Set $\avt = 2 \dt \ew{E}_{h}^{k}$ in (\ref{manipconvectivespin}).
We rewrite $A_2$ as
\begin{align*}
A_2 = - 2\inertiamom\dt\tril_{h}\!\lp\bv{u}^k,\bv{R}^k,\ew{E}_{h}^{k}\rp
- 2\dt\inertiamom\tril_{h}\!\lp\bv{E}_{h}^k,\bv{w}_h^k,\ew{E}_{h}^{k}\rp
- 2\dt\inertiamom\tril_{h}\!\lp \bv{S}^k,\bv{w}_h^k,\ew{E}_{h}^{k}\rp 
= A_{21} + A_{22} + A_{23} \, . 
\end{align*}
Since $\bv{u}^k$ is solenoidal the bound on $A_{21}$ proceeds as that of $A_{12}$, 
whereas \eqref{fifthdiscineq} gives control on $A_{22}$ and $A_{23}$:
\[
  A_2 \leq \frac{3c_1 \dt}{7} \|\gradv{}\ew{E}_{h}^{k}\|_{\ltwods}^2 
  + \frac{7 M \inertiamom^2 \dt}{c_1} \left(
    \|\pew^{k}\|_{\ltwods}^2 + \|\bv{E}_{h}^k\|_{\ltwods}^2 + \|\peu^k\|_{\ltwods}^2
  \right).
\]

The bound on $A_3$ begins by noticing that 
$ \bv{w}^{k} - \bv{W}_h^{k-1} = \inc \bv{w}^{k} + \pew^{k-1} + \ew{E}_{h}^{k-1} $. The integration by parts formula \eqref{intbypartscurl} then yields
\[
  A_3 = 4\dt \nu_r \lp \inc\bv{w}^{k},\curl{E}_{h}^{k} \rp 
+ 4\dt \nu_r \lp \pew^{k-1},\curl{E}_{h}^{k} \rp 
+ 4\dt \nu_r \lp \ew{E}_{h}^{k-1},\curl{E}_{h}^{k} \rp \, ,
\]
whence
\begin{align*}
\begin{split}
A_3 &\leq 4\dt \nu_r \|\inc\bv{w}^{k}\|_{\ltwods} \|\gradv{E}_{h}^{k}\|_{\ltwods} 
+ 4\dt \nu_r \|\pew^{k-1}\|_{\ltwods} \|\gradv{E}_{h}^{k}\| 
+ 4\dt \nu_r \|\ew{E}_{h}^{k-1}\|_{\ltwods} \|\gradv{E}_{h}^{k}\|_{\ltwods}  \\
&\leq \frac{\nunot \dt}{3} \|\gradv{E}_{h}^{k}\|_{\ltwods}^2 
+ \frac{36 \nu_r^2\dt}{\nunot} \|\inc\bv{w}^{k}\|_{\ltwods}^2 
+ \frac{36 \nu_r^2\dt}{\nunot} \|\pew^{k-1}\|_{\ltwods}^2
+ \frac{36 \nu_r^2\dt}{\nunot} \|\ew{E}_{h}^{k-1}\|_{\ltwods}^2.
\end{split}
\end{align*}
The term $\|\inc\bv{w}^{k}\|_{\ltwods}^2$ can be bounded
similarly to \eqref{incrementestimate}.

The bound on $A_4$ follows the same lines as those of $A_3$:
\[
A_4
= 4 \dt \nu_r \lp \peu^k,\curl{}\ew{E}_{h}^{k}\rp 
+ 4 \dt \nu_r \lp \bv{E}_{h}^k,\curl{}\ew{E}_{h}^{k}\rp
\leq \frac{2 c_1 \dt}{7} \|\gradv{}\ew{E}_{h}^{k}\|_{\ltwods}^2
+ \frac{28\nu_r^2\dt}{c_1} \|\peu^k\|_{\ltwods}^2
+ \frac{28\nu_r^2\dt}{c_1} \|\bv{E}_{h}^k\|_{\ltwods}^2.
\]
The last two terms $A_5$ and $A_6$ can be easily bounded as follows
\begin{align*}
A_5 &= - 2 \lp \inc\bv{S}^k , \bv{E}_{h}^{k} \rp 
- 2\inertiamom \lp\inc\bv{R}^k ,\ew{E}_{h}^{k}\rp 
\leq \frac{\nunot \dt}{9} \|\gradv{E}_{h}^{k}\|_{\ltwods}^2 
+ \frac{9 C_p^2}{\nunot \dt} \|\inc\bv{S}^k\|_{\ltwods}^2
+ \frac{c_1 \dt}{7} \|\gradv{}\ew{E}_{h}^{k}\|_{\ltwods}^2 
+ \frac{7 C_p^2 \inertiamom^2 }{c_1 \dt} \|\inc\bv{R}^k\|_{\ltwods}^2 \, , 
\end{align*}
and
\begin{align*}
A_6 = 2 \dt \lp\mathcal{R}_\bv{u}^k,\bv{E}_{h}^{k}\rp
+ 2 \dt \inertiamom \lp \mathcal{R}_\bv{w}^k,\ew{E}_{h}^{k}\rp
\leq \frac{\nunot \dt}{9} \|\gradv{E}_{h}^{k}\|_{\ltwods}^2 
+ \frac{9 C_p^2\dt}{\nunot} \|\mathcal{R}_\bv{u}^k\|_{\ltwods}^2
+ \frac{c_1 \dt}{7} \|\gradv{}\ew{E}_{h}^{k}\|_{\ltwods}^2 
+ \frac{7 C_p^2 \inertiamom^2 \dt }{c_1} \|\mathcal{R}_\bv{w}^k\|_{\ltwods}^2.
\end{align*}

The interpolation errors are bounded by \eqref{approximab} which, in conjunction with
\eqref{initialregass}, also implies
\begin{align}\label{errordifferences}
\begin{split}
\|\inc\bv{S}^k\|_{\ltwods} + h \|\inc\gradv{S}^k\|_{\ltwods}
&\leq C \dt \, h^{\po+1} \spatialconstantone(\bv{u}_t,p_t) \\
\|\inc\bv{R}^k\|_{\ltwods} + h \|\inc\gradv{R}^k\|_{\ltwods} &\leq C \dt \, h^{\po+1}
\spatialconstanttwo(\bv{w}_t).
\end{split}
\end{align}
Assumption \eqref{initialregass} also gives an estimate on the truncation errors $\mathcal{R}_\bv{u}^k $
and $\mathcal{R}_\bv{w}^k$,
\begin{align}\label{residualestimate}
\begin{split}
\|\mathcal{R}_\bv{u}^k \|_{\ltwods}^2 \leq \frac{\dt}{3} \int_{t^{k-1}}^{t^{k}}
\|\bv{u}_{tt}\|_{\ltwods}^2 \, dt \ \ \ , \ \ \
\|\mathcal{R}_\bv{w}^k \|_{\ltwods}^2 \leq \frac{\dt}{3} \int_{t^{k-1}}^{t^{k}}
\|\bv{w}_{tt}\|_{\ltwods}^2 \, dt.
\end{split}
\end{align}
Inserting the estimates above for $A_i$, $1 \leq i \leq 6$, into \eqref{globalerrorstructure},  summing in $k$ and application of Gr\"onwall inequality concludes the proof.
\end{proof}

\begin{remark}[Smallness assumption on $\dt$]
Condition \eqref{Kvalue} does not depend on the space discretization parameter $h$. It does depend, however, on the constants $M$ and $\mathcal{M}$ defined in \eqref{eq:Mestiamtes}; this is standard for Navier-Stokes. In addition, this estimate depends on the quotients $\nu_r^2/\nunot$ and $\nu_r^2/c_1$, which gives an indication of how strong the coupling between linear and angular
momentum is. 
\end{remark}

\subsection{Error Estimates for the Discrete Time Derivative}
When dealing with the Navier-Stokes equations, it is well-known (see, for instance, \cite{GuerQuart}) that in order to derive optimal error estimates for the pressure in $\ell^2(\ltwo)$ one must first obtain estimates on the discrete time derivative of the velocity, which is the main reason for the additional regularity requested in \eqref{additionalreg}. Our analysis is no exception, and this is additionally complicated by the fact that we must obtain error estimates for the derivatives of the linear and angular velocities. However, it is important to point out that it is possible derive an error estimate.

Applying the increment operator $\inc$, defined in \eqref{backwarddifferenceop}, to the equations that govern the approximation errors and proceeding as in the proof of 
Proposition~\ref{prop:schstable} we conclude that the discrete time derivatives
$\dt^{-1} \inc \bv{E}_h^\dt$ and $\dt^{-1} \inc \mathcal{E}_h^\dt$ satisfy an energy identity  
similar to \eqref{globalerrorstructure}, namely, 
\begin{align}
\label{seconddifference}
\begin{split}
\|\dt^{-1}\inc\bv{E}_h^{k}\|_{\ltwods}^2 &+ (\inertiamom + 8 {\nu}_r \dt) \|\dt^{-1}\inc\ew{E}_{h}^{k}\|_{\ltwods}^2 
- \|\dt^{-1}\inc\bv{E}_h^{k-1}\|_{\ltwods}^2 
- \inertiamom \|\dt^{-1}\inc\ew{E}_{h}^{k-1}\|_{\ltwods}^2 
+ 2 \nunot \dt \| \dt^{-1} \inc\nabla\bv{E}_{h}^{k} \|_{\ltwods}^2  \\
&+ 2 c_1 \dt  \|\dt^{-1} \inc\gradv{}\ew{E}_h^{k}\|_{}^2 
+ \|\dt^{-1}\inc^2\bv{E}_h^{k}\|_{\ltwods}^2
+ \inertiamom \|\dt^{-1}\inc^2\ew{E}_{h}^{k}\|_{\ltwods}^2 
+ 2 c_2 \dt \|\dt^{-1} \diver{} \inc\ew{E}_h^{k}\|_{\ltwods}^2
= \sum_{i = 1}^5 F_i
\end{split}
\end{align}
where  
\begin{align*}
F_1 &= 2 \tril_h\lp\bv{U}_h^{k-1},\bv{U}_h^{k},\dt^{-1} \inc\bv{E}_h^{k}\rp
- 2 \tril_h\lp\bv{u}^k,\bv{u}^k,\dt^{-1} \inc\bv{E}_h^{k}\rp
- 2 \tril_h\lp\bv{U}_h^{k-2},\bv{U}_h^{k-1},\dt^{-1} \inc\bv{E}_h^{k}\rp
+ 2 \tril_h\lp\bv{u}^{k-1},\bv{u}^{k-1},\dt^{-1} \inc\bv{E}_h^{k}\rp \\
F_2 &= 2 \inertiamom\tril_h\lp\bv{U}_h^{k},\bv{W}_h^{k},\dt^{-1} \inc\ew{E}_h^{k} \rp
- 2 \inertiamom\tril_h\lp\bv{u}^k,\bv{w}^k,\dt^{-1} \inc\ew{E}_h^{k}\rp
- 2 \inertiamom \tril_h\lp\bv{U}_h^{k-1},\bv{W}_h^{k-1},\dt^{-1} \inc\ew{E}_h^{k}\rp
+ 2 \inertiamom \tril_h\lp\bv{u}^{k-1},\bv{w}^{k-1},\dt^{-1} \inc\ew{E}_h^{k}\rp \\
F_3 &= 4 \nu_r \lp \curl{}\inc\bv{w}^{k} - \curl{}\inc\bv{W}_h^{k-1},\dt^{-1} \inc\bv{E}_h^{k} \rp +
4 \nu_r \lp \curl{}\inc\bv{u}^k - \curl{}\inc\bv{U}_h^{k},\dt^{-1} \inc\ew{E}_h^{k}\rp \\
F_4 &= - 2 \dt^{-1} \lp\inc^2\bv{S}^k, \dt^{-1}\inc\bv{E}_h^{k}\rp 
- 2 \inertiamom \dt^{-1} \lp \inc^2\bv{R}^k,\dt^{-1} \inc\ew{E}_h^{k}\rp \\
F_5 &= 2 \lp\inc\mathcal{R}_\bv{u}^k,\dt^{-1} \inc\bv{E}_h^{k}\rp 
+ 2\inertiamom \lp\inc\mathcal{R}_\bv{w}^k,\dt^{-1} \inc\ew{E}_h^{k}\rp.
\end{align*} 

A bound on these terms then yields a bound on the discrete time derivatives. This is the content of the following result.

\begin{theorem}[Error estimate for the discrete time derivatives]\label{thm:boudnderiv}
Assume \eqref{additionalreg}. If
\begin{equation}
\label{smallnessreq}
h^{-1/2} \|\bv{E}_h^\dt\|_{\ell^\infty(\ltwods)} \qquad \text{and} \qquad
h^{-1/2} \|\ew{E}_h^\dt\|_{\ell^\infty(\ltwods)}
\end{equation}
are sufficiently small, then
\begin{align}\label{linfltwoesterror}
  \left\| \frac{\inc\bv{E}_h^\dt}{\dt} \right \|_{\ell^\infty(\ltwods)}
  + \left\| \frac{\inc\ew{E}_h^\dt}{\dt} \right\|_{\ell^\infty(\ltwods)}
  \leq C( \dt + h^l).
\end{align}
\end{theorem}
\begin{proof}
In analogy to Theorem~\ref{maintheorem}, it suffices to bound the
residual terms $\lb F_{i}\rb_{i=1}^5$. The proof is rather technical and tedious, and consists of careful manipulations of these five terms. Take the difference of \eqref{convectivemanipulation} for two consecutive time-steps, which allows us to write $F_1$ as the sum of six terms $\lb F_{1i}\rb_{i=1}^6$:
\begin{align*}
F_{11} &= - 2 \tril(\inc \bv{u}^{k},\bv{u}^k,\dt^{-1} \inc\bv{E}_h^{k}) + 2 \tril(\inc
\bv{u}^{k-1},\bv{u}^{k-1},\dt^{-1} \inc\bv{E}_h^{k})\\
F_{12} &= - 2 \tril(\bv{u}^{k-1},\bv{S}^k,\dt^{-1} \inc\bv{E}_h^{k}) + 2
\tril(\bv{u}^{k-2},\bv{S}^{k-1},\dt^{-1} \inc\bv{E}_h^{k}) \\
F_{13} &= - 2 \tril_h(\bv{S}^{k-1},\bv{u}_h^k, \dt^{-1} \inc\bv{E}_h^{k}) + 2
\tril_h(\bv{S}^{k-2},\bv{u}_h^{k-1},\dt^{-1} \inc\bv{E}_h^{k})\\
F_{14} &= - 2 \tril_h(\bv{E}_h^{k-1},\bv{u}_h^k, \dt^{-1} \inc\bv{E}_h^{k}) + 2
\tril_h(\bv{E}_h^{k-2},\bv{u}_h^{k-1}, \dt^{-1} \inc\bv{E}_h^{k})\\
F_{15} &= 2 \tril_h(\bv{E}_h^{k-1}, \bv{E}_h^{k}, \dt^{-1} \inc\bv{E}_h^{k}) 
- 2 \tril_h(\bv{E}_h^{k-2}, \bv{E}_h^{k-1}, \dt^{-1} \inc\bv{E}_h^{k})\\
F_{16} &= - 2 \tril_h(\bv{u}_h^{k-1}, \bv{E}_h^{k}, \dt^{-1} \inc\bv{E}_h^{k}) + 2
\tril_h(\bv{u}_h^{k-2}, \bv{E}_h^{k-1}, \dt^{-1} \inc\bv{E}_h^{k}) \, . 
\end{align*}
Using the linearity and skew-symmetry of the trilinear form, these six terms can be appropriately rewritten and bounded using \eqref{firstineq}-\eqref{thirdineq} and \eqref{firstdiscineq}-\eqref{additionalineq} to get
\begin{align*}
F_{11} &= 2 \tril(\inc \bv{u}^{k},\dt^{-1} \inc\bv{E}_h^{k},\inc\bv{u}^k) 
+ 2 \tril(\inc^2 \bv{u}^{k},\dt^{-1} \inc\bv{E}_h^{k},\bv{u}^{k-1}) \\
&\leq \frac{\nunot \dt}{14} \|\dt^{-1} \inc\gradv{E}_h^{k}\|_{\ltwods}^2 
+ \frac{C}{\nunot \dt} \|\inc \gradv{u}^{k}\|_{\ltwods}^4 
+ \frac{\nunot \dt}{14} \|\dt^{-1} \inc\gradv{E}_h^{k}\|_{\ltwods} 
+ \frac{M}{\nunot \dt} \|\inc^2 \bv{u}^{k}\|_{\ltwods}^2 \, , \\
F_{12} &= 2 \tril(\bv{u}^{k-1},\dt^{-1} \inc\bv{E}_h^{k},\inc\bv{S}^k) + 2
\tril(\inc\bv{u}^{k-1},\dt^{-1} \inc\bv{E}_h^{k},\bv{S}^{k-1}) \\ 
&\leq \frac{\nunot \dt}{14} \|\dt^{-1} \inc\gradv{E}_h^{k}\|_{\ltwods}^2 
+ \frac{M}{\nunot \dt}  \|\inc\bv{S}^k\|_{\ltwods}^2 
+ \frac{\nunot \dt}{14} \|\dt^{-1} \inc\gradv{E}_h^{k}\|_{\ltwods}^2 
+ \frac{C}{\nunot \dt} \|\inc\gradv{u}^{k-1}\|_{\ltwods}^2 \|\gradv{S}^{k-1}\|_{\ltwods}^2 \, , \\
F_{13} &= - 2 \tril_h(\inc\bv{S}^{k-1},\bv{u}_h^k, \dt^{-1} \inc\bv{E}_h^{k}) - 2
\tril_h(\bv{S}^{k-2},\inc\bv{u}_h^{k-1},\dt^{-1} \inc\bv{E}_h^{k}) \\
&\leq \frac{\nunot \dt}{14} \|\dt^{-1} \inc\gradv{E}_h^{k}\|_{\ltwods}^2 
+ \frac{M}{\nunot \dt} \|\inc\bv{S}^{k-1}\|_{\ltwods}^2
+ \frac{\nunot \dt}{14} \|\dt^{-1} \inc\gradv{E}_h^{k}\|_{\ltwods}^2 
+ \frac{C}{\nunot \dt} \|\inc\gradv{u}_h^{k-1}\|_{\ltwods}^2 \|\gradv{S}^{k-2}\|_{\ltwods}^2 \, , \\
F_{14} &= - 2 \dt \tril_h(\dt^{-1}\inc\bv{E}_h^{k-1},\bv{u}_h^k, \dt^{-1} \inc\bv{E}_h^{k}) - 2
\tril_h(\bv{E}_h^{k-2},\inc\bv{u}_h^{k}, \dt^{-1} \inc\bv{E}_h^{k}) \\ 
&\leq \frac{\nunot \dt}{14} \|\dt^{-1} \inc\gradv{E}_h^{k}\|_{\ltwods}^2 
+ \frac{M\dt^2}{\nunot \dt} \|\dt^{-1}\inc\bv{E}_h^{k-1}\|_{\ltwods}^2
+ \frac{\nunot \dt}{14} \|\dt^{-1} \inc\gradv{E}_h^{k}\|_{\ltwods}^2 
+ \frac{C}{\nunot \dt}
\|\inc\gradv{u}_h^{k}\|_{\ltwods}^2 \|\gradv{E}_h^{k-2}\|_{\ltwods}^2 \, , \\
F_{15} &= 2 \dt \tril_h(\dt^{-1} \inc\bv{E}_h^{k-1},\bv{E}_h^{k-1},\dt^{-1} \inc\bv{E}_h^{k}) 
\leq  \frac{C \|\bv{E}_h^{k-1}\|_{\ltwods}}{h^{\nicefrac{1}{2}}} \dt
\alp  \|\dt^{-1} \inc\gradv{E}_h^{k-1}\|_{\ltwods}^2 +  \|\dt^{-1} \inc\gradv{E}_h^{k}\|_{\ltwods}^2
\arp \, , \\
F_{16} &= - 2 \tril_h(\inc\bv{u}_h^{k-1}, \bv{E}_h^{k}, \dt^{-1} \inc\bv{E}_h^{k}) 
\leq \frac{\nunot \dt}{14} \|\dt^{-1} \inc\gradv{E}_h^{k}\|_{\ltwods}^2 
+ \frac{C}{\nunot \dt} \|\inc\gradv{u}_h^{k-1}\|_{\ltwods}^2 \|\gradv{E}_h^{k}\|_{\ltwods}^2 \, .
\end{align*}
Similarly, applying $\inc$ to \eqref{manipconvectivespin}, $F_2$ can be expressed as the sum of five terms $\lb F_{2i}\rb_{i=1}^5$:
\begin{align*}
F_{21} &= - 2 \inertiamom \tril_h\lp\bv{u}^{k},\bv{R}^{k}, \dt^{-1} \inc\ew{E}_h^{k} \rp 
+ 2 \inertiamom \tril_h\lp\bv{u}^{k-1},\bv{R}^{k-1}, \dt^{-1} \inc\ew{E}_h^{k} \rp \, , \\
F_{22} &= 2 \inertiamom \tril_h\lp\bv{E}_h^{k},\ew{E}_h^{k},\dt^{-1} \inc\ew{E}_h^{k} \rp
- 2 \inertiamom \tril_h\lp\bv{E}_h^{k-1},\ew{E}_h^{k-1},\dt^{-1} \inc\ew{E}_h^{k} \rp \, , \\
F_{23} &= - 2 \inertiamom \tril_h\lp\bv{E}_h^{k},\bv{w}_h^{k},\dt^{-1} \inc\ew{E}_h^{k} \rp
+ 2 \inertiamom \tril_h\lp\bv{E}_h^{k-1},\bv{w}_h^{k-1},\dt^{-1} \inc\ew{E}_h^{k} \rp \, , \\
F_{24} &= - 2 \inertiamom \tril_h\lp\bv{S}^{k},\bv{w}_h^{k},\dt^{-1} \inc\ew{E}_h^{k} \rp
+ 2 \inertiamom \tril_h\lp\bv{S}^{k-1},\bv{w}_h^{k-1},\dt^{-1} \inc\ew{E}_h^{k} \rp \, , \\
F_{25} &= - 2 \inertiamom \tril_h\lp\bv{u}_h^{k},\ew{E}_h^{k},\dt^{-1} \inc\ew{E}_h^{k} \rp 
+ 2 \inertiamom \tril_h\lp\bv{u}_h^{k-1},\ew{E}_h^{k-1},\dt^{-1} \inc\ew{E}_h^{k} \rp \, .
\end{align*}
We now bound each of these terms separately 
\begin{align*}
F_{21} &= - 2 \inertiamom\tril_h\lp \inc\bv{u}^{k},\bv{R}^{k}, \dt^{-1} \inc\ew{E}_h^{k} \rp 
+ 2 \inertiamom \tril_h\lp \bv{u}^{k-1},\inc\bv{R}^{k}, \dt^{-1} \inc\ew{E}_h^{k} \rp \\
&\leq \frac{c_1 \dt}{12} \|\dt^{-1} \inc\gradv{}\ew{E}_h^{k}\|_{\ltwods}^2 
+ \frac{C \inertiamom^2}{c_1 \dt} \|\inc\gradv{u}^{k}\|_{\ltwods}^2
\|\gradv{R}^{k}\|_{\ltwods}^2
+ \frac{c_1 \dt}{12} \|\dt^{-1} \inc\gradv{}\ew{E}_h^{k}\|_{\ltwods}^2 
+ \frac{M\inertiamom^2}{c_1 \dt} \|\inc\bv{R}^{k}\|_{\ltwods}^2 \, , \\
F_{22} &= 2 \inertiamom \dt \tril_h\lp \dt^{-1}\inc\bv{E}_h^{k},\ew{E}_h^{k}, \dt^{-1}
\inc\ew{E}_h^{k} \rp 
\leq \frac{C \|\ew{E}_h^{k}\|_{\ltwods}}{h^{\nicefrac{1}{2}}} \inertiamom \dt
\alp  \|\dt^{-1}\inc\gradv{E}_h^{k}\|_{\ltwods}^2 + \|\dt^{-1}
\inc\gradv{}\ew{E}_h^{k}\|_{\ltwods}^2 \arp \, , \\
F_{23} &= - 2 \inertiamom\tril_h\lp \bv{E}_h^{k},\inc\bv{w}_h^{k},\dt^{-1} \inc\ew{E}_h^{k} \rp
+ 2 \inertiamom \dt \tril_h\lp \dt^{-1} \inc\bv{E}_h^{k},\bv{w}_h^{k-1},\dt^{-1} \inc\ew{E}_h^{k}
\rp \\
&\leq \frac{c_1 \dt}{12} \|\dt^{-1} \inc\gradv{}\ew{E}_h^{k}\|_{\ltwods}^2
+ \frac{C \inertiamom^2}{c_1 \dt} \|\inc\gradv{w}_h^{k}\|_{\ltwods}^2
\|\gradv{E}_h^{k}\|_{\ltwods}^2  
+ \frac{c_1 \dt}{12} \|\dt^{-1} \inc\gradv{}\ew{E}_h^{k}\|_{\ltwods}^2 
+ \frac{\mathcal{M}\inertiamom^2\dt^2}{c_1 \dt} \|\dt^{-1} \inc\bv{E}_h^{k}\|_{\ltwods}^2 \, , \\
F_{24} &= - 2 \inertiamom \tril_h\lp\inc\bv{S}^{k},\bv{w}_h^{k},\dt^{-1} \inc\ew{E}_h^{k} \rp
- 2 \inertiamom \tril_h\lp\bv{S}^{k-1},\inc\bv{w}_h^{k},\dt^{-1} \inc\ew{E}_h^{k} \rp \\ 
&\leq \frac{c_1 \dt}{12} \|\dt^{-1} \inc\gradv{}\ew{E}_h^{k}\|_{\ltwods}^2
+ \frac{\mathcal{M} \inertiamom^2}{\varepsilon_2} \|\inc\bv{S}^{k}\|_{\ltwods}^2
+ \frac{c_1 \dt}{12} \|\dt^{-1} \inc\gradv{}\ew{E}_h^{k}\|_{\ltwods}^2 
+ \frac{C \inertiamom^2}{c_1 \dt} \|\inc\gradv{w}_h^{k}\|_{\ltwods}^2
\|\gradv{S}^{k-1}\|_{\ltwods}^2 \, , \\
F_{25} &= - 2 \inertiamom \tril_h\lp \inc\bv{u}_h^{k},\ew{E}_h^{k},\dt^{-1} \inc\ew{E}_h^{k} \rp 
- 2 \inertiamom \tril_h\lp\bv{u}_h^{k-1},\inc\ew{E}_h^{k},\dt^{-1} \inc\ew{E}_h^{k} \rp \\
&\leq \frac{c_1 \dt}{12} \|\dt^{-1} \inc\gradv{}\ew{E}_h^{k}\|_{\ltwods}^2 
+ \frac{C \inertiamom^2}{c_1 \dt} \|\inc\gradv{u}_h^{k}\|_{\ltwods}^2
\|\gradv{}\ew{E}_h^{k}\|_{\ltwods}^2
+ \frac{c_1 \dt}{12} \|\dt^{-1} \inc\gradv{}\ew{E}_h^{k}\|_{\ltwods}^2 
+ \frac{M\inertiamom^2}{c_1 \dt} \|\inc\ew{E}_h^{k}\|_{\ltwods}^2 \, .
\end{align*}
By virtue of \eqref{intbypartscurl}, $F_3$ can be estimated as follows:
\begin{align*}
F_3 &= 4 \nu_r \lp \inc^2\bv{w}^{k} + \inc\bv{R}^{k-1} + \inc\ew{E}_h^{k-1},\dt^{-1}
\curl{}\inc\bv{E}_h^{k} \rp 
+ 4 \nu_r \lp \inc\bv{S}^k + \inc\bv{E}_h^{k},\dt^{-1} \curl{}\inc\ew{E}_h^{k}\rp \\
&\leq \frac{3}{14}\nunot \dt \|\dt^{-1} \inc\gradv{}\bv{E}_h^{k}\|_{\ltwods}^2 
+ \frac{56 \nu_r^2}{\dt \nunot } 
\alp \|\inc^2\bv{w}^{k}\|_{\ltwods}^2 
+ \|\inc\bv{R}^{k-1}\|_{\ltwods}^2 +
\|\inc\ew{E}_h^{k-1}\|_{\ltwods}^2 \arp \\
&+ \frac{c_1 \dt}{6} \|\inc\gradv{}\ew{E}_h^{k}\|_{\ltwods} 
+ \frac{48 \nu_r^2}{\dt c_1} 
\alp \|\inc\bv{S}^k\|_{\ltwods}^2 + \|\inc\bv{E}_h^{k}\|_{\ltwods}^2 \arp \, . 
\end{align*}
The last two terms $F_4$ and $F_5$ require no further manipulation and result in 
\begin{align*}
F_4 &\leq \frac{\nunot \dt}{14} \|\dt^{-1}\inc\gradv{E}_h^{k}\|_{\ltwods}^2  +
\frac{14 C_p^2\dt^{-2}}{\nunot \dt} \|\inc^2\bv{S}^k\|_{\ltwods}^2 
+ \frac{c_1 \dt}{12} \|\dt^{-1}\inc\gradv{}\ew{E}_h^{k}\|_{\ltwods}^2  
+ \frac{12 C_p^2 \dt^{-3}\inertiamom^2}{c_1} \|\inc^2\bv{R}^k\|_{\ltwods}^2 \, , \\
F_5 &\leq \frac{\nunot \dt}{14} \|\dt^{-1}\inc\gradv{E}_h^{k}\|_{\ltwods}^2  
+ \frac{14 C_p^2}{\nunot \dt} \|\inc\mathcal{R}_\bv{u}^k\|_{\ltwods}^2
+ \frac{c_1 \dt}{12} \|\dt^{-1}\inc\gradv{}\ew{E}_h^{k}\|_{\ltwods}^2  
+ \frac{12 C_p^2 \inertiamom^2}{c_1 \dt} \|\inc\mathcal{R}_\bv{w}^k\|_{\ltwods}^2 \, .
\end{align*}
Collecting all the estimates for $\|\dt^{-1}\inc\gradv{E}_h^{k}\|_{\ltwods}^2$ and $\|\dt^{-1}\inc\gradv{}\ew{E}_h^{k}\|_{\ltwods}^2 $, and using assumption \eqref{smallnessreq}, we get
\[
\frac{2 C \|\bv{E}_h^k\|_{\ltwods}}{h^{1/2}}  \dt +
\frac{C \inertiamom \|\ew{E}_h^k\|_{\ltwods} }{h^{1/2}} \dt \leq \nunot \dt
\qquad
\frac{2 C \inertiamom \|\ew{E}_h^k\|_{\ltwods} }{h^{1/2}}  \dt \leq c_1 \dt. 
\]
These conditions allow for cancellation of the problematic terms $F_{15}$ and $F_{22}$
with the fifth and sixth terms on the left hand side of \eqref{seconddifference}. Finally, summation of the energy identity \eqref{seconddifference} and application of Gr\"onwall inequality lead to \eqref{linfltwoesterror}.
\end{proof}

\begin{remark}[Smallness assumption] The error estimate \eqref{finalLinfL2error} shows that \eqref{smallnessreq} is actually a condition of the form 
\begin{align}\label{sortofCFL}
\dt h^{-1/2} \leq C_{s} < \infty
\end{align}
for a small enough constant $C_s$.
This requirement is not a special characteristic of our method but rather a recurrent feature in the analysis of schemes for the Navier-Stokes equations. See, for instance \cite{GuerQuart,HeyRann}.
\end{remark}

\subsection{Error Estimates for the Pressure}
\label{errorestimatespressure}
The control on the derivatives of the velocities provided by Theorem~\ref{thm:boudnderiv} enables us to obtain error estimates for the pressure. To do so, it is crucial that the discrete spaces are compatible in the sense of \eqref{discreteinfsup}. This is the idea behind the following result.

\begin{theorem}[Error estimate for the pressure] If \eqref{additionalreg} and \eqref{linfltwoesterror} are valid, then the following estimate holds
\begin{equation}
\label{pressureerror}
  \| e^\dt \|_{\ell^2(L^2)} \leq C \left( \dt + h^\po \right).
\end{equation}
\end{theorem}
\begin{proof}
As already mentioned, the approximation errors $\bv{E}_h^\dt$ and $e_h^\dt$ are actually solutions to \eqref{diffeqfullydiscv} with a special right hand side composed of
consistency terms. Condition \eqref{discreteinfsup} then allows us to write
\begin{align}
\label{pressureerr1}
\begin{split}
\beta^* \|e_h^k\|_{\ltwo} &\leq \sup_{\lvt\in \mathbb{V}_h}
\frac{(e_h^k,\diver{}\lvt)}{\ \|\lvt\|_{\honeds}} 
\leq \sum_i^6 B_i,
\end{split}
\end{align}
where
\begin{align*}
&B_1 = \sup_{\lvt\in \mathbb{V}_h} \frac{\dt^{-1}\lp\inc\bv{E}_h^{k},\lvt\rp}{\ \|\lvt\|_{\honeds}} \, , 
&&B_2 = \nunot \sup_{\lvt\in \mathbb{V}_h} \frac{\lp\gradv{E}_h^k, \gradv{}\lvt\rp}{\
\|\lvt\|_{\honeds}}  \, , \\
&B_3 = \sup_{\lvt\in \mathbb{V}_h} \frac{\tril_h(\bv{u}^k,\bv{u}^k,\lvt) 
- \tril_h(\bv{U}_h^{k-1},\bv{U}_h^{k},\lvt)}{\ \|\lvt\|_{\honeds}}  \, ,
&&B_4 = 2 \nu_r \sup_{\lvt\in \mathbb{V}_h} \frac{(\curl{W}_h^{k-1} - \curl{w}^k,\lvt)}{\
\|\lvt\|_{\honeds}}  \, , \\
&B_5 = \sup_{\lvt\in \mathbb{V}_h} \frac{\dt^{-1} \lp\inc\bv{S}^k, \lvt\rp}{\ \ \
\|\lvt\|_{\honeds}}  \, ,
&&B_6 = \sup_{\lvt\in \mathbb{V}_h} \frac{(\mathcal{R}_\bv{u}^k,\lvt)}{\ \ \ \|\lvt\|_{\honeds}} \, .
\end{align*}
So that it suffices to provide suitable bounds for each one of these terms.

We readily have, for $B_1$ and $B_2$, that
\begin{align*}
B_1 &= \sup_{\lvt\in \mathbb{V}_h} \frac{\dt^{-1}\lp\inc\bv{E}_h^{k},\lvt\rp}{\ \|\lvt\|_{\honeds}} 
\lesssim  \bigl\| \dt^{-1} \inc\bv{E}_h^{k} \bigl\|_{\ltwods} \ , \ \ \
B_2 \lesssim \|\gradv{E}_h^k\|_{\ltwods}.
\end{align*}

Identity \eqref{convectivemanipulation} can be used to express the numerator of $B_3$ as
\begin{align*}
  B_3 &\leq
  \sup_{\lvt\in \mathbb{V}_h} \frac{\tril_h(\inc \bv{u}^{k},\bv{u}^k,\lvt)}{\|\lvt\|_{\honeds}}
  + \sup_{\lvt\in \mathbb{V}_h} \frac{\tril_h(\bv{u}^{k-1},\bv{S}^k,\lvt)}{\|\lvt\|_{\honeds}}
  + \sup_{\lvt\in \mathbb{V}_h} \frac{\tril_h(\bv{E}^{k-1},\bv{u}_h^k, \lvt)}{\|\lvt\|_{\honeds}} \\
  &+  \sup_{\lvt\in \mathbb{V}_h} \frac{\tril_h(\bv{E}_h^{k-1},\bv{E}_h^k, \lvt)}{\|\lvt\|_{\honeds}}
  + \sup_{\lvt\in \mathbb{V}_h} \frac{\tril_h(\bv{u}_h^{k-1}, \bv{E}_h^{k}, \lvt) }{\|\lvt\|_{\honeds}}
  = \sum_{i=1}^5 B_{3i}.
\end{align*} 
Inequality \eqref{secondineqstar} and the regularity assumptions \eqref{initialregass} imply
\[
B_{31} = \sup_{\lvt\in \mathbb{V}_h} \frac{\tril_h(\inc \bv{u}^{k},\bv{u}^k,\lvt)}{\
\|\lvt\|_{\honeds}}
\lesssim \|\bv{u}^k\|_{\htwos} \|\inc \bv{u}^{k}\|_{\ltwods} \lesssim \|\inc \bv{u}^{k}\|_{\ltwods}.
\]
To bound
$B_{32}$, $B_{33}$ and $B_{35}$ we use inequality \eqref{firstdiscineq}, the stability
\eqref{projectionstability} of the projectors and the regularity assumptions
\eqref{initialregass},
\begin{align*}
B_{32} &= \sup_{\lvt\in \mathbb{V}_h} \frac{\tril_h(\bv{u}^{k-1},\bv{S}^k,\lvt)}{\
\|\lvt\|_{\honeds}} \lesssim \|\gradv{u}^{k-1}\|_{\ltwods}   \|\gradv{S}^k\|_{\ltwods} 
\lesssim \|\gradv{S}^k\|_{\ltwods} \, , \\
B_{33} &= \sup_{\lvt\in \mathbb{V}_h} \frac{\tril_h(\bv{E}^{k-1},\bv{u}_h^k, \lvt)}{\
\|\lvt\|_{\honeds}} \lesssim \|\gradv{E}^{k-1}\|_{\ltwods} \|\gradv{u}_h^{k}\|_{\ltwods} 
\lesssim \|\gradv{E}^{k-1}\|_{\ltwods} \, , \\
B_{35} &= \sup_{\lvt\in \mathbb{V}_h} \frac{\tril_h(\bv{u}_h^{k-1}, \bv{E}_h^{k}, \lvt) }{\
\|\lvt\|_{\honeds}} \lesssim \|\gradv{u}_h^{k-1} \| \|\gradv{E}_h^{k}\|_{\ltwods} 
\lesssim \|\gradv{E}_h^{k}\|_{\ltwods}. 
\end{align*}
The first inequality in \eqref{additionalineq} yields
\begin{align*}
B_{34} &= \sup_{\lvt\in \mathbb{V}_h} \frac{\tril_h(\bv{E}_h^{k-1}, \bv{E}_h^{k}, \lvt)}{\
\|\lvt\|_{\honeds}} 
\lesssim h^{-\nicefrac{1}{2}} \, \|\bv{E}_h^{k-1}\|_{\ltwods} \|\gradv{E}_h^k\|_{\ltwods}.
\end{align*}
In conclusion, we have proved the bound
\[
|B_3| \lesssim \|\inc \bv{u}^{k}\|_{\ltwods} + \|\gradv{S}^k\|_{\ltwods}
+ \|\gradv{S}^{k-1}\|_{\ltwods} + \|\gradv{E}_h^{k-1}\|_{\ltwods}
+ h^{-\nicefrac{1}{2}} \, \|\bv{E}_h^{k-1}\|_{\ltwods} \|\gradv{E}_h^k\|_{\ltwods}
+ \|\gradv{E}_h^{k}\|_{\ltwods}.
\]

Integrating by parts as in \eqref{intbypartscurl}, we infer that
\[
B_4 = 2 \nu_r \sup_{\lvt\in \mathbb{V}_h} \frac{(\curl{W}_h^{k-1} - \curl{w}^k,\lvt)}{\
\|\lvt\|_{\honeds}} = 
\sup_{\bv{v}\in \hzerod} \frac{(\inc\bv{w}^k + \ew{E}^{k-1},\curl{}\lvt)}{\ \|\lvt\|_{\honeds}}
\lesssim \|\inc\bv{w}^k\|_{\ltwods} + \|\ew{E}^{k-1}\|_{\ltwods} .
\]

Finally, we see that
\[
B_5 = \sup_{\lvt\in \mathbb{V}_h} \frac{\dt^{-1} \lp\inc\bv{S}^k , \lvt\rp}{\ \ \
\|\lvt\|_{\honeds}} \lesssim \dt^{-1} \|\inc\bv{S}^k\|_{\ltwods},
\qquad
B_6 = \sup_{\bv{v}\in \hzerod} \frac{(\mathcal{R}_\bv{u}^k,\bv{v})}{\ \|\bv{v}\|_{\honeds}} 
\lesssim \|\mathcal{R}_\bv{u}^k\|_{\ltwods}.
\] 

It suffices now to realize that all the bounds involve consistency, interpolation or approximation errors and that they all have the right order. This concludes the proof.
\end{proof}

\section{A Second Order Scheme}
\label{sec:2ndorder}

Let us present a second order scheme for the solution of \eqref{eq:microNS} and show its stability properties. We work in the setting of \S\ref{sub:timedisc} and \S\ref{finitelementspace}. We first recall a three-term recursion inequality originally shown in \cite{MR2802553}, which is instrumental to show stability.

\begin{proposition}[Three term recursion]
\label{prop:ThreeTermInduction}
The three term recursion equation
\begin{equation}
  3x^{k+1}-4x^k+x^{k-1} = g^{k+1},\quad \forall k\geq 1,
\label{eq:diff-eq-stokes}
\end{equation}
has the following general solution
\[
  x^\nu = c_1 + \frac{c_2}{3^\nu} + \sum_{l=2}^\nu\frac1{3^{\nu+1-l}}\sum_{s=2}^l g^s,
    \qquad c_1,\ c_2\in \mathbb{R}.
\]

Let $\{y^k\}_{k\ge0}$ be the solution to the three term recursion inequality
\[
  3y^{k+1}-4y^k+y^{k-1} \leq g^{k+1},\quad \forall k\geq 1,
\]
with initial data $y^0$ and $y^1.$ If $\{x^k\}_{k\ge0}$ is the
solution to \eqref{eq:diff-eq-stokes} with initial data $x^0=y^0$ and
$x^1=y^1,$ then the following estimate holds
\[
  y^\nu \leq x^\nu, \quad \forall \nu\geq0.
\]
\end{proposition}


For $\lb y^k \rb_{k\geq0}$ let $\inc_{^-}^2 y^k$ denote the second order backward difference, i.e.
\[
\inc_{^-}^2 y^k = \tfrac{1}{2} ( 3 y^{k+1} - 4 y^k + y^{k-2} ) \ \ \ \ \forall k \geq 2.
\]

Let us now describe the scheme. We begin with an initialization step, in which we set
\[
  \left( \bv{U}_h^k, P_h^k, \bv{W}_h^k \right) = \left( \bv{u}_h^k, p_h^k, \bv{w}_h^k \right),
  \quad k=0,1.
\]
In other words, we compute the Stokes and elliptic-like projections of the initial data and the solution on the first time step. This initialization is only for ease of presentation as it clearly requires knowledge of the exact solution. In practice one can compute the projection of the initial data and then perform one step with the first order scheme of Section~\ref{sec:Scheme}.

We march in time, for $k=2,\ldots,K$, as follows:

\noindent \textbf{Linear Momentum:} Find $(\bv{U}_h^k,P_h^k) \in \polV_h \times \polQ_h$ that solves
\begin{subequations}
  \begin{align}
    \label{eq:lmom2ndorder}
  \alp \frac{ \inc_{^-}^2 \bv{U}_h^k  }{\dt}, \lvt \arp 
  + \nunot \lp \gradv{U}_h^{k},\gradv{}\lvt \rp
  + \tril_h\lp \bv{U}_h^{k,\star},\bv{U}_h^{k},\lvt\rp - \lp P_h^{k}, \diver{}\lvt \rp 
  &= 2 \nu_r \lp \curl{W}_h^{k,\star},\lvt \rp + \lp \bv{f}^{\,k}, \lvt \rp,     \\
    \label{incomp}
    \lp \ptf, \diver{U}_h^k\rp &= 0 \, ,    
  \end{align}
\end{subequations}
for all $ \lvt \in \mathbb{V}_h$, $\ptf \in \mathbb{Q}_h$, where, for a time-discrete function $\phi^\dt$, we introduced the second order extrapolation
  \begin{equation}
  \label{eq:2ndorderextr}
    \phi^{k,\star} = 2\phi^{k-1} - \phi^{k-2}.
\end{equation}
  
\noindent \textbf{Angular Momentum:} Compute $\bv{W}_h^k \in \polV_h$, solution of 
\begin{multline}
\label{eq:amom2ndorder}
  \inertiamom   \alp \frac{ \inc_{^-}^2 \bv{W}_h^k  }{\dt},\avt \arp 
  + c_1 \lp \gradv{W}_h^{k},\nabla \avt \rp 
  + \inertiamom \tril_h\lp\bv{U}_h^{k},\bv{W}_h^{k},\avt\rp + \\
  + c_2 \lp \diver{W}_h^{k},\diver{}\avt \rp + 4 {\nu}_r \lp \bv{W}_h^{k}, \avt \rp =
  2\nu_r \lp \curl{U}_h^{k}, \avt \rp + \lp \bv{g}^{k},\avt \rp \, .
\end{multline}
for all $\avt \in \mathbb{V}_h$.

This scheme turns out to be almost unconditionally stable, as shown in the following result. To avoid irrelevant technicalities, we assume that $\bv{f}^\dt = \bv{g}^\dt = 0$.

\begin{theorem}[Stability of second order scheme]
\label{thm:stab2ndorder}
Assume that the time step satisfies
\begin{equation}
\label{eq:conddt2ndorder}
  \dt \leq \frac{\inertiamom\nu}{8\nu_r^2}.
\end{equation}
Then, the sequence $\{ \bv{U}_h^\dt, \bv{W}_h^\dt, P_h^\dt \} \subset [\polV_h]^2 \times \polQ_h$, solution of \eqref{eq:lmom2ndorder}--\eqref{eq:amom2ndorder}, satisfies
\[
  \| \bv{U}_h^\dt \|_{\ell^\infty(\bv{L}^2)} + \|\bv{W}_h^\dt \|_{\ell^\infty(\bv{L}^2)}
  +
  \| \gradv{U}_h^\dt \|_{\ell^2(\bv{L}^2)} + \|\gradv{W}_h^\dt \|_{\ell^2(\bv{L}^2)}
  \leq C,
\]
where the constant depends on the material parameters and the values 
of $\{ \bv{U}_h^k, \bv{W}_h^k, P_h^k \}$ for $k=0,1$, but does not depend on the discretization parameters.
\end{theorem}
\begin{proof}
We combine the techniques used to prove Proposition~\ref{prop:schstable} and Theorem~5.1 of \cite{MR2802553}. We begin by setting $\lvt = 4 \dt \bv{U}_h^k$ in \eqref{eq:lmom2ndorder} and $\avt = 4\dt \bv{W}_h^k$ in \eqref{eq:amom2ndorder} and adding the result. Using \eqref{incomp}, we obtain
\begin{multline*}
  3 y^k - 4 y^{k-1} + y^{k-2}
  +2\inc\left( \| \inc \bv{U}_h^k \|_{\bv{L^2}}^2 
      + \inertiamom \| \inc\bv{W}_h^k \|_{\bv{L^2}}^2  \right)
  +\| \inc^2 \bv{U}_h^k \|_{\bv{L^2}}^2 + \inertiamom \| \inc^2 \bv{W}_h^k \|_{\bv{L^2}}^2
  + 4\dt \left( \nunot \| \gradv{U}_h^k \|_{\bv{L^2}}^2 + c_1 \| \gradv{W}_h^k \|_{\bv{L^2}}^2\right) \\
  + 4 c_2 \dt \| \diver{W}_h^k \|_{\bv{L^2}}^2 + 16 \nu_r \dt \|\bv{W}_h^k \|_{\bv{L^2}}^2 
  = 8 \nu_r \dt \lp \curl{U}_h^k,  2 \bv{W}_h^k - \inc^2 \bv{W}_h^k \rp,
\end{multline*}
where
\[
  y^k = \| \bv{U}_h^k \|_{\bv{L^2}}^2 + \inertiamom \| \bv{W}_h^k \|_{\bv{L^2}}^2 \, . 
\]
Here we used the identity
\[
  2a^k(3a^k-4a^{k-1}+a^{k-2}) = 3|a^k|^2 - 4 |a^{k-1}|^2 + |a^{k-2}|^2 
  + 2\inc|\inc a^k|^2 + |\inc^2 a^k|^2 \, . 
\]
and, to produce the right hand side, we integrated by parts using \eqref{intbypartscurl} and employed the equality
\[
  \phi^k + \phi^{\star,k} = \phi^k + 2 \phi^{k-1} - \phi^{k-2} = 2\phi^k - \inc^2 \phi^k,
\]
which is a consequence of \eqref{eq:2ndorderextr}. Using \eqref{curlineq} we obtain 
\[
  8 \nu_r \dt \lp \curl{U}_h^k,  2 \bv{W}_h^k - \inc^2 \bv{W}_h^k \rp \leq
  16 \nu_r \dt \|\bv{W}_h^k \|_{\bv{L^2}}^2 + 4\nu_r \dt \| \gradv{U}_h^k \|_{\bv{L^2}}^2
  + \inertiamom \| \inc^2 \bv{W}_h^k \|_{\bv{L^2}}^2 
  + \frac{16 \nu_r^2 \dt^2}{\inertiamom} \| \gradv{U}_h^k \|_{\bv{L^2}}^2.
\]
Since $\nunot = \nu + \nu_r$ assumption \eqref{eq:conddt2ndorder} yields
\[
  4 \nu \dt - \frac{16 \nu_r^2 \dt^2}{\inertiamom} = 
  4\nu \dt \left( 1 - \frac{4\nu_r^2\dt}{\inertiamom \nu} \right)
  \geq 2\nu\dt \, .
\]
The estimates of Proposition~\ref{prop:ThreeTermInduction} imply the assertion.
\end{proof}

\begin{remark}[Time step constraint]
Notice that the constraint on the time step \eqref{eq:conddt2ndorder}, necessary for stability, is meaningful. First of all, the quantity on the right hand side has units of time. In addition, it is consistent with the fact that, for the classical Navier-Stokes equations (that is $\nu_r = 0$) no constraints are necessary for the stability of a second order semi-implicit discretization.
\end{remark}

\section{Numerical Validation}
\label{sec:numericalvalidation}
We now present a numerical validation of our error estimates. The implementation has been
carried out with the help of the \texttt{deal.II} library, see \cite{BHK2007,DealIIReference}. We use the lowest order Taylor-Hood elements, that is $\mathbb{Q}_2/\mathbb{Q}_1$, so that $\po=2$. The arising linear systems have been solved with the direct solver \texttt{UMFPACK}$^\copyright$. 

Consider a square domain $\Omega = (0,1)^2 \subset \mathbb{R}^2$, and a smooth divergence-free linear velocity, pressure, and angular velocity defined by 
\begin{align*}
\bv{u}(x,y,t) &= \left( \sin(2 \pi x + t) \, \sin(2 \pi y + t),
              \cos(2 \pi x + t) \, \cos(2 \pi y + t)  \right)^\intercal, \\
p(x,y,t) &= \sin(2 \pi (x - y) + t), \\   
\bv{w}(x,y,t) &= \sin(2 \pi x + t) \, \sin(2 \pi y + t). \, 
\end{align*}

To verify the $\ell^2(\honed)$ error for the velocity and the $\ell^2(\ltwo)$ error for the
pressure we fix the relationship $\dt = h^2$, and consider a sequence of meshes with
$h = 2^{-i}$ for $2 \leq i \leq 6$. The corresponding errors are displayed in Figure~\ref{figura1}, thereby showing clearly the predicted convergence rates.

\begin{figure}
\begin{center}
\includegraphics[scale=0.55]{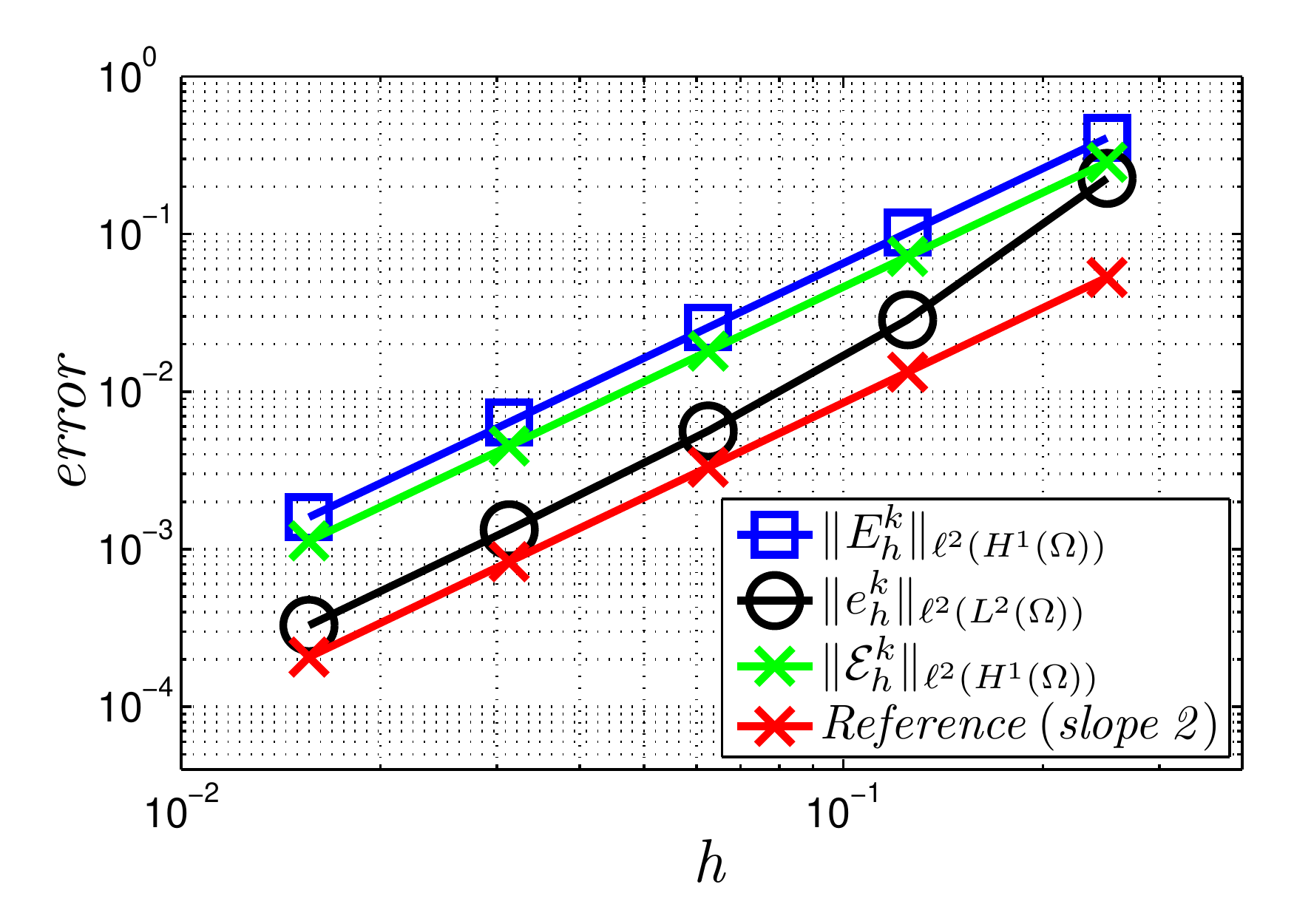}
\caption{$\ell^2(\honed)$ error of the velocities and $\ell^2(\ltwo)$ error of the
pressure with respect to mesh size. The axes are in logarithmic scale.}
\label{figura1}
\end{center}
\end{figure}

To validate the $\ell^{\infty}(\ltwod)$ error of the velocities we fix the
relationship $\dt = h^3$, and consider the same sequence of meshes.
The corresponding errors are depicted in Figure~\ref{figura2} and exhibit the expected optimal rates.

\begin{figure}
\begin{center}
\includegraphics[scale=0.55]{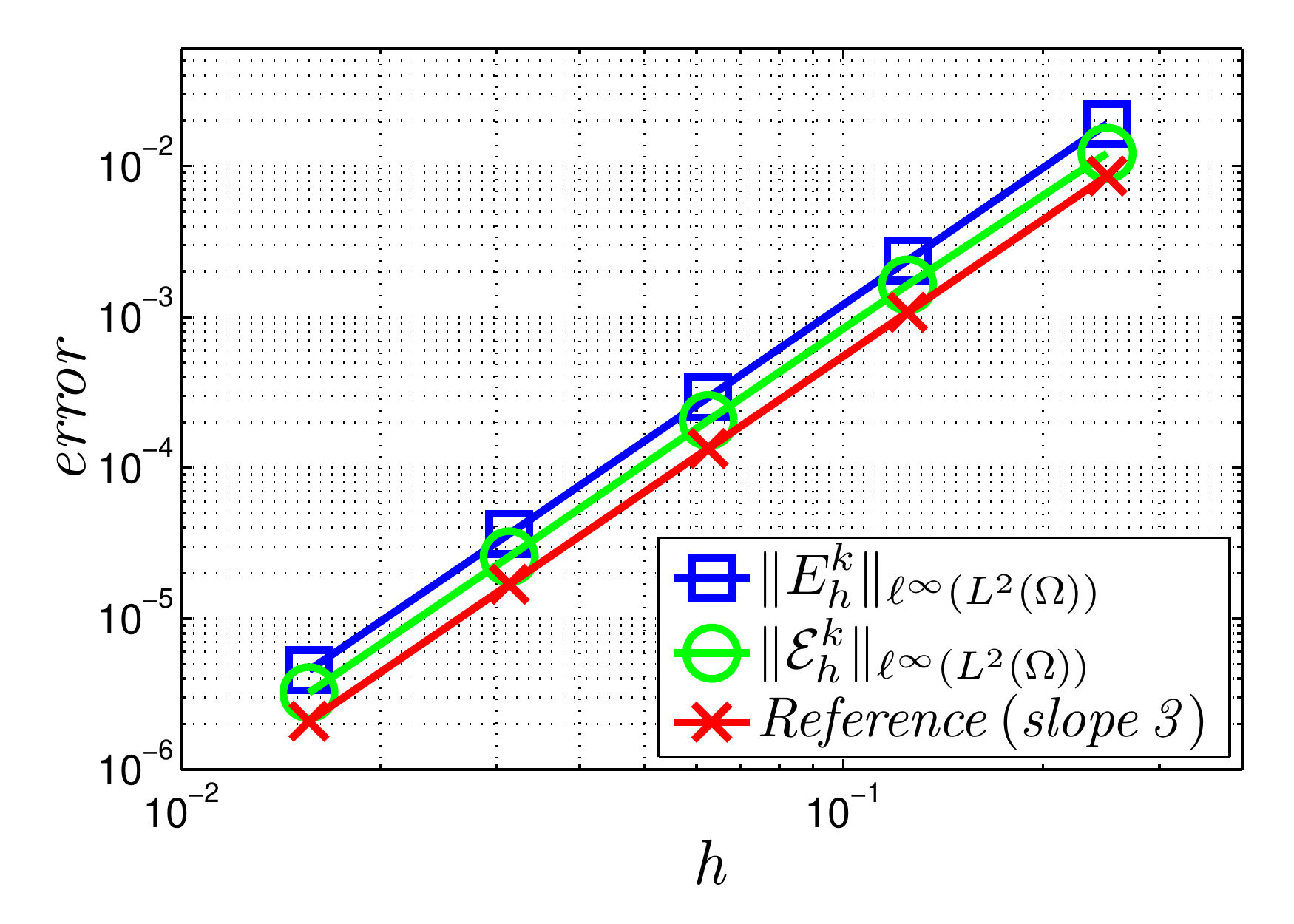}
\caption{$\ell^{\infty}(\ltwod)$ error of the velocities with respect to mesh size.
The axes are in logarithmic scale.}
\label{figura2}
\end{center}
\end{figure}

\section{Conclusions and Perspectives}
We have presented a first order, fully discrete semi-implicit scheme for the MNSE which is unconditionally stable and possesses optimal convergence rates in time and space. The scheme is semi-implicit, therefore it only involves, at every time-step, the solution of linear systems. In addition, the equations of linear and angular momentum are decoupled, which makes the implementation simpler and the scheme more efficient. To further decouple the unknowns, fractional time-stepping techniques can be incorporated, and we believe that their analysis shall not present difficulties beyond those already encountered in this work. 

We have also presented a formally second order scheme which is almost unconditionally stable and shares similar properties to the first order scheme, i.e., it is semi-implicit, decouples the linear and angular velocities and it can be easily simplified further with fractional time stepping techniques. The error analysis of such a scheme will be reported elsewhere, where in addition we will explore whether the stability condition is indeed a requirement of our scheme, or an artifact of our methods of proof.
 
The idea of pumping micropolar fluid through excitation of the spin equation was explored by testing a simple family of forcing terms $\bv{g}$. It was observed computationally that the regimes of effective pumping and reverse pumping regimes are not well separated. In other words, very similar forcing terms $\bv{g}$ can induce very different effects in the velocity profile, or even opposite effects (reverse direction of the net flow).

The most challenging extension of this work is towards the solution of the equations of ferrohydrodynamics: the MNSE with \eqref{MagmicroNS} coupled with the magnetostatic equations. The design, analysis and implementation of a scheme for this problem requires techniques and ideas well beyond those presented here, but will allow for more interesting and realistic simulations. This is part of future developments. 


\bibliographystyle{siam}
\bibliography{biblio}

\end{document}